\documentclass[11pt,reqno,a4paper]{amsart}
\usepackage{amssymb}
\usepackage{lineno}
\usepackage{a4}

\usepackage{amsmath,color,ifthen} 
\usepackage{bbm} 
\usepackage[numbers]{natbib}
\usepackage{enumerate}
\usepackage{graphicx}
\usepackage{hyperref}

\newtheorem{theo}{Theorem}
\newtheorem{lem}{Lemma}
\newtheorem{defi}{Definition}
\newtheorem{prop}{Proposition}
\newtheorem{cor}{Corollary}
\theoremstyle{remark}
\newtheorem{remark}{Remark}
\newtheorem{example}{Example}

\theoremstyle{definition}
\newtheorem{assumption}{Assumption}

\def\band{\mathrm{b}}
\def\diffop{\mathbf{\Delta}}
 
\def\E{\mathbb{E}} 

\def\R{\mathbb{R}}
\def\Z{\mathbb{Z}} 
\def\C{\mathbb{C}} 
\def\calT{\mathcal{T}} 
\def\calN{\mathcal{N}} 
\def\bD{\mathbf{D}} 
\def\be{\mathbf{e}} 
\def\1{\mathbbm{1}} 
\def\rme{\mathrm{e}}
\def\rmi{\mathrm{i}}
\def\rmd{\mathrm{d}}
\def\bw{\mathbf{w}}
\def\constantePsiD{\mathbf{K}}

\newenvironment{enum_W}
  {%
  \setlength{\leftmargini}{4em}\begin{enumerate}}
  {\end{enumerate}}

\begin{document}

\title{Locally stationary long memory estimation}
\date{April 8, 2010}

\author{Fran\c{c}ois Roueff} 
\address{Institut Telecom, Telecom Paris, CNRS LTCI, 46 rue Barrault, 75634 Paris Cedex 13,
  France}
\email{roueff@telecom-paristech.fr}
\author{Rainer von Sachs}
\address{Institut de statistique, Universit\'e catholique de Louvain, Voie du Roman Pays, 20,
B-1348 Louvain-la-Neuve, Belgium}
\email{rvs@uclouvain.be}

\begin{abstract}
  There exists a wide literature on parametrically or semi-parametrically
  modeling strongly dependent time series using a long-memory parameter $d$,
  including more recent work on wavelet estimation. As a generalization of
  these latter approaches, in this work we allow the long-memory parameter $d$ to
  be varying over time. We adopt a semi-parametric approach in order to avoid
  fitting a time-varying parametric model, such as tvARFIMA, to the observed
  data. We study the asymptotic behavior of a local log-regression wavelet
  estimator of the time-dependent $d$. Both simulations and a real data example
  complete our work on providing a fairly general approach.
\end{abstract}


\maketitle

\section{Introduction}\label{sec:1}
There is a long tradition of modelling the phenomenon of long-range dependence
in observed data that show a strong persistence of their correlations by
long-memory processes. Such data can typically be found in the applied sciences
such as hydrology, geophysics, climatology and telecommunication
(e.g. teletraffic data) but recently also in economics and in finance, e.g. for
modelling (realized) volatility of exchange rate data or stocks. The literature
on stationary long-memory processes is huge (see e.g. the references in the
recent survey paper~\cite{fay-mouline-roueff-taqqu-2009}), and we concentrate
here on the discussion of long-range dependence resulting from a singularity of
the spectral density at zero frequency - corresponding to a slow,
i.e. polynomial, decay of the autocorrelation of the data. Although a lot of
(earlier) work started from a parametric approach, using e.g. the celebrated
ARFIMA-like models, it occurs that since the seminal work by P. Robinson (see
\cite{robinson:1995:GSE,robinson:1995:GPH}), the semi-parametric approach is
known to be more robust against model misspecification: instead of using a
parametric filter describing both the singularity of the spectral density at
zero frequency and the ARMA-dynamics of the short memory part, only the
singular behavior of the spectrum at zero is modelled by the long-memory
parameter, $d$ say, whereas the short memory part remains completely
non-parametric.

Driven by the empirical observation that the correlation structure of observed
(weakly or strongly dependent) data can change over time, there is a also a
growing literature on modelling departures from covariance-stationarity, mainly
restricted to the short-range dependent case. One prominent approach, that we
adopt in this paper, too, is the model of local stationarity, introduced by a
series of papers by R. Dahlhaus (\cite{dahlhaus96,dahlhaus:1997,dahlhaus2000}):
in a non-parametric set-up, the spectral structure of the underlying stochastic
process is allowed to be smoothly varying over time. Of course, time-varying
linear processes (of ARMA type) arise as a subclass of these locally stationary
processes. In order to come up with a rigorous asymptotic theory of consistency
and inference, the time-dependence of the spectral density $f(u,\lambda)$ of
such locally stationary processes is modelled to be in rescaled time $u\ \in
[0,1]$, leading to a problem of non-parametric curve estimation: increasing the
sample size $T$ of the observed time series does no more mean to look into the
future but to dispose of more and more observations to identify
$f(t/T,\lambda)$ locally around the ``reference'' rescaled time point $u
\approx t/T$.

In the aforementioned spirit of semi-parametric modelling, and in contrast to
the parametric approach of \cite{beran:2009}, one of the very few existing
approaches on time-varying long-memory modelling, we consider in this paper a
locally stationary long-range dependent process with a singularity in the
spectral density at zero frequency which is parameterized by a time-varying
long-memory parameter $d=d(u)$, $u\in[0,1]$, i.e. defined in rescaled time. Our
approach is a true generalization of the stationary approach in that the latter
corresponds to a time-constant $d$ for our locally stationary model. As in the
case of \cite{moulines:roueff:taqqu:2007:jtsa}, the long memory parameter is
estimated by a log-regression of a series of wavelet scalograms (estimated
wavelet variances per scale by summing the squared wavelet coefficients per
scale over location) onto a range of scales (corresponding to the low frequency
range of the spectrum). Although wavelets do not improve the estimation of $d$
in the standard stationary context $-1/2<d<1/2$, their use is of interest in
various practical situations (presence of trends, under and over-differenced
series leading to $d\geq1/2$ and $d\leq-1/2$ respectively), see details
in~\cite{fay-mouline-roueff-taqqu-2009}.
However, in our work now the challenge is to \emph{localize} the estimation of
the no more constant parameter $d$. Wavelets are favorable in this situation
since, in contrast to a Fourier analysis, they are well localized both over
time and frequency, i.e. scale. The localization is achieved by smoothing over
time the series of squared wavelet coefficients on each of the coarse scales
which enter into the log-regression, giving raise to a \emph{local
  scalogram}. We propose both a more traditional method based on two-sided
kernels and also a recursive scheme of one-sided smoothing weights, adapted to
the end point of the observation period.

The model studied in this paper arises naturally in the now long history of
time series modelling in presence of persistent memory.  A survey on this
subject is provided in \cite{samo:2006}.  In Chapter~3 of this reference, it is
recalled how long memory and non stationarity have been used as concurrent
modelling approaches, in particular for financial data, see \textit{e.g.}
\cite{mikosch-starica-2004}.  Long memory modeling for financial data goes back
to \cite{ding-granger-engle1993}. Originally investigated on absolute powers of
stock returns, long memory models are currently widely used for realized
volatility data since they were proven efficient for forecasting purposes in
this context in \cite{ABDL-2003}. It is interesting to note that only 3 years
after \cite{ding-granger-engle1993}, the need for time varying long memory
parameters was pointed out in \cite{ding-granger-1996}, see Section~5 and in
particular Figure~6 where an estimated $d$ is plotted evolving between the
values $0.358$ and $0.714$ over a 60 years long period. In this reference two
approaches are suggested for coping with a time varying $d$, namely, a
stochastically evolving $d$ or a regime switching between, say, two values of
$d$.  These two approaches have been developed, respectively, in
\cite{ray-tsay-2002} and \cite{haldrup-nielsen-2006} (see also
\cite{guegan:lu:2009} where singularities in the spectral density at
frequencies different from zero are considered in a piecewise stationary
context).

As outlined earlier the alternative approach developed by Dahlhaus for locally
stationary (short memory) processes is quite appealing as it allows a
meaningful asymptotic study of the estimators.  It was recently used for
volatility estimation using time varying (short memory) non-linear processes,
see \cite{dahlhaus-subbarao-2006,fryzlewicz-sapatinas-subbarao-2008}.  The
first attempt to adapt Dahlhaus's approach in the presence of long memory
appears to be the unpublished preprint \cite{jensen:whitcher:2000}. The authors
use the log linear relationship of the local variance of the maximum overlap
discrete wavelet transform and their scaling parameter, plus a localization
with a rectangular window in coefficient domain, to estimate the time-varying
long memory parameter. However, the asymptotic analysis of the proposed
estimator is not provided. Although it is not essential in their analysis, the
considered model is a tvARFIMA$(q,d,r)$, see our examples below.  An asymptotic
analysis is provided in \cite{beran:2009} for a different estimator applied to
the same model.  Roughly speaking, the standard way to estimate a time varying
parameter $d(u)$ at rescaled time $u=t/T\in(0,1)$ of a locally stationary
process is to use that, for a sample size $T$, $Tb$ observations around time
index $t$ approximately behave as a stationary sample as the bandwidth
parameter $b\to0$.  In \cite[Theorems~1]{beran:2009} the proposed estimator of
the time varying long memory parameter $d$ is claimed to satisfy a central
limit theorem at rate $(Tb)^{-1/2}$. On the other hand Theorem 2 in
\cite{beran:2009} says that the bias is of order $b^2$ if $d$ is two times
continuously differentiable. Such results are somewhat similar to that for
estimating the time varying parameter of a locally stationary short memory
process, see \cite{moulines-priouret-roueff-2005}, or, in a more general
fashion, Example 3.6 in \cite{dahlhaus-2009}. Hence the presence of long memory
in \cite{beran:2009} seems not to affect the estimation rates. It can be
explained by the parametric approach for the correlation structure of the
observed locally stationary time series in that the filter in the linear
(although infinite) locally autoregressive representation of the process is
completely determined by a finite-dimensional parameter. It follows that only a
finite number of local correlations are needed to determine the local parameter
$d$. In other words the estimation of $d(u)$ can be obtained from an analysis
over a fixed set of frequencies taken away from zero. This would however induce a
high sensitivity to model misspecification. In contrast, as usual in
semi-parametric vs parametric estimation, we want our approach to be robust against
model misspecification. To impose this robustness in the semi-parametric
model, the memory parameter $d$ is disconnected from the spectral content out
of the origin so that short range correlations does not carry any information
about this parameter. 
This fully justifies the semi-parametric context even though it is more
involved as it necessitates a \textit{low frequency} analysis (where the long
memory behavior occurs) which, at first sight, seems contradictory to the
\textit{local} stationarity framework. In fact, this contradiction is inherent
to any locally stationary model which rely on a compromise between
stationarity, which appears at small scales, and analyzing bandwidth, which
requires a large scale to decrease the randomness in the data. As a consequence
practical applications of these models require very long data sets.
Our results will prove that this apparent contradiction can still be overcome
for locally stationary long memory models, but with some price to pay on the
rate of estimation (although we are unable to prove that this price is
optimal).  It is not as surprising as it may appear. To understand why,
consider a piecewise stationary context where a finite number of regime
switching times occur over the observation sample. One clearly sees that the
long memory over each stationary segment can be estimated at the same rate as
in the stationary context. As we will see, the picture is more complicated in a
locally stationary context but it can still be handled.  More precisely, we
will show that looking at low frequencies is allowed in a locally stationary
model but with an additional cost on the rate of convergence depending on how
small the frequency used for the analysis is. We believe that such theoretical
results are crucial for the practical estimation of the time varying long
memory parameter as they demonstrate the viability of such an approach while
indicating that it should be applied with care.

Summarizing our results, the rest of the paper is organized as follows. In
Section~\ref{sec:2}, we give the technical details of our locally stationary
long memory model of semi-parametric type and give a series of examples of
processes falling into this model.  In Section~\ref{sec:3} we define our
estimators based on wavelet analysis for which we briefly recall the wavelet
set-up.
We define the local scalogram which is at the heart of our wavelet based
estimators. We also prepare our technique of stationary approximation by
defining what we call the approximating stationary \emph{tangent process} and
its wavelet spectrum, the \emph{local wavelet spectrum},
as well as the pseudo-estimator \emph{tangent scalogram}. We finish this
section by discussing a series of smoothing weights, one- and two-sided
kernels, which fulfill our given assumptions.  The asymptotic properties of our
proposed estimators are stated in the following Section~\ref{sec:4}.
We derive a mean-square approximation of the local scalogram through the
tangent scalogram (Proposition~\ref{prop:local-scalogram}), followed by a
control of the mean square error of the scalogram as an estimator of the local
wavelet spectrum (Theorem~\ref{theo:rate-param-scalo}) and a CLT for the
tangent scalogram (Theorem~\ref{thm:clt-tang-scal}), which finally allows us to
derive a CLT for the local scalogram (Corollary~\ref{cor:clt-local-scalogram}).
The results on the asymptotic behavior of the estimator of $d(u)$ are then
obtained: Corollary~\ref{cor:rate-param-est} provides the rate of convergence
and Theorem~\ref{theo:clt-param-est} the asymptotic normality.
We pursue the paper by Section~\ref{sec:5} on numerical examples, first
simulating some ARFIMA process with a time-varying $d$ and comparing the
performance of the two-sided (rectangular) kernel with the recursive weight
scheme. Second, we apply the kernel estimator to a series of realized log
volatilities (see also \cite{bauwens:wang:2008}), namely of the exchange rate
of the YEN versus USD, from June 1986 to September 2004. We conclude in
Section~\ref{sec:6} before an appendix section presents all technical details
of our derivations including our proofs.

\section{Model set-up and examples}\label{sec:2}
Define the difference operator $[\diffop X]_k = X_k - X_{k-1}$ and $\diffop^p$ recursively.
This will allow $d(u)$ to take values up to $p+1/2$ in the following model.

We adapt the approach of~\cite{dahlhaus96} to the case where the spectral
density is allowed to have a singularity at the zero frequency.  Let us fix
$p=0,1,2,\dots$ and $A^0_{t,T}(\lambda)$ be an array of $L^2([-\pi,\pi])$
functions with real-valued Fourier coefficients.  Let $\{X_{t,T}\}$ be an 
array of real-valued random variables such that
\begin{equation}
  \label{eq:LocStatProcSpectralExt}
\diffop^p X_{t,T} = 
\int_{-\pi}^\pi A^0_{t,T}(\lambda)\ \rme^{\rmi\lambda t}\;\rmd Z(\lambda)\ , \qquad t=1,\ldots T, \ T \geq 1\ ,
\end{equation}
where $\rmd Z(\lambda)$ is the spectral representation of a centered weak white
noise with unit variance,
\begin{equation}
  \label{eq:epsZ}
  \varepsilon_t= \int_{-\pi}^\pi\rme^{\rmi\lambda t}\;\rmd Z(\lambda)
,\quad t\in\Z \;,
\end{equation}
hence  $Z(\lambda)$ is a Hermitian complex valued
process with weakly stationary orthogonal increments on $[-\pi,\pi]$.
We further assume that there exist a function $A(u,\lambda)$ in
$L^2([0,1]\times[-\pi,\pi])$ and two constants $c>0$ and $D<1/2$ such that
\begin{equation}
  \label{eq:AssumAtimeSmooth2}
\left|A^0_{t,T}(\lambda)-A(t/T,\lambda)\right|\leq c\; T^{-1} \; |\lambda|^{-D}\,,\quad 1\leq t\leq T,\,-\pi\leq\lambda\leq\pi \;,
\end{equation}
and
\begin{equation}
  \label{eq:AssumAtimeSmooth}
\left|A(u;\lambda)-A(v,\lambda)\right|\leq c\; |v-u| \; |\lambda|^{-D}\,,\quad 0\leq u,v\leq1,\,-\pi\leq\lambda\leq\pi\;.  
\end{equation}
These correspond to the definition of locally stationary processes introduced
in~\cite{dahlhaus96} but with the term $|\lambda|^{-D}$ added in the upper
bound to allow a singularity at the zero frequency.  
Relations~(\ref{eq:LocStatProcSpectralExt}), (\ref{eq:AssumAtimeSmooth2})
and~(\ref{eq:AssumAtimeSmooth}) give rise to the following time-varying
\emph{generalized} spectral density of $\{X_{t,T}\}$
\begin{equation}
  \label{eq:genLocSpecDens}
f(u,\lambda) = |1- \rme^{-\rmi\lambda}|^{-2p}\;\left|A(u;\lambda)\right|^2 \; .  
\end{equation}
The first multiplicative factor in the right-hand side
of~(\ref{eq:genLocSpecDens}) corresponds to the operator $\diffop^p$ in the
left-hand side of~(\ref{eq:LocStatProcSpectralExt}). We now focus on
time-varying generalized spectral densities exhibiting a \emph{memory parameter}
at zero frequency.
\begin{defi}
  We say that the process $\{X_{t,T},\;t=1, \ldots, T, \ T \geq 1\}$ has
  \emph{local memory parameter} $d(u)\in(-\infty,p+1/2)$ at time $u\in [0,1]$
  if it
  satisfies~(\ref{eq:LocStatProcSpectralExt}),~(\ref{eq:AssumAtimeSmooth2})
  and~(\ref{eq:AssumAtimeSmooth}) and its generalized spectral density
  $f(u,\lambda)$ defined by~(\ref{eq:genLocSpecDens}) satisfies the following
  semi-parametric type condition:
\begin{equation}
  \label{eq:genLocSpecDensSemiParam}
f(u,\lambda) = |1- \rme^{-\rmi\lambda}|^{-2d(u)}\ f^*(u,\lambda)\ ,
\end{equation}
with $f^*(u,0)>0$ and
\begin{equation}
  \label{eq:genLocSpecDensSemiParamCond}
 |f^*(u,\lambda)-f^*(u,0)|\leq C\;f^*(u,0)\;|\lambda|^\beta,\;\lambda\in[-\pi,\pi] \;,
\end{equation}
where $C>0$ and $\beta\in(0,2]$.
\end{defi}
The assumption on the model is summarized hereafter.
\begin{assumption}
  \label{ass:model_lin}
  The array $\{X_{t,T}\}$ of real-valued random variables has local memory
  parameter $d(u)\in(-\infty,p+1/2)$ at time $u\in [0,1]$. Moreover
  $\{\varepsilon_t\}$ in~(\ref{eq:epsZ}) is a weak white noise such that
  $\E[\varepsilon_0]=0$, $\mathrm{Var}(\varepsilon_0)=1$, $\E[\varepsilon_t^4]$
  is finite for all $t\in\Z$ and the fourth-order cumulants of its spectral
  representation $\rmd Z(\lambda)$ satisfy
  \begin{equation}
    \label{eq:cum-spectral}
    \mathrm{Cum}\left(\rmd
      Z(\lambda_k),\,1\leq
      k\leq4\right)=\hat{\kappa}_4(\lambda)\;\rmd\mu(\lambda),
    \quad\lambda=(\lambda_k)_{1\leq k\leq 4}\in[-\pi,\pi]^4\;,
  \end{equation}
  where $\hat{\kappa}_4(\lambda)=\hat{\kappa}_4(\lambda_1,\lambda_2,\lambda_3)$
  is a bounded function defined on $[-\pi,\pi]^3$, and $\mu$ is defined as the
  measure on $[-\pi,\pi]^4$ such that, for any $(2\pi)$-periodic functions
  $A_k$, $1\leq k\leq 4$,
  \begin{equation}
    \label{eq:muDef}
    \int_{[-\pi,\pi]^4} \prod_{k=1}^4A_k(\lambda_k)\;\rmd\mu(\lambda)=
\int_{[-\pi,\pi]^3}A_4(-\lambda_1-\lambda_2-\lambda_3)\,
\prod_{k=1}^3A_k(\lambda_k)\;\rmd\lambda\;.
  \end{equation}
\end{assumption}

Assumption~(\ref{eq:cum-spectral}) is standard for linear representations of
time series and was used by Dahlhaus (for cumulants of all orders) in the
original definition of locally stationary processes in \cite{dahlhaus:1997}.
The measure $\mu$ defined by~(\ref{eq:muDef}) is sometimes denoted as
$\rmd\mu(\lambda)=\eta(\lambda_1+\dots+\lambda_4)\rmd\lambda_1\dots\rmd\lambda_4$,
where $\eta$ is the $(2\pi)$-periodic Dirac comb, see \textit{e.g.}
\cite{dahlhaus:1997,dahlhaus-neumann-2001}.
An immediate consequence of~(\ref{eq:cum-spectral}) is the following bound of
fourth-order cumulants, for any set of $(2\pi)$-periodic functions $A_k$, $1\leq k\leq
4$, 
\begin{equation}
  \label{eq:cumBound}
  \left|\mathrm{Cum}\left(\int_{-\pi}^\pi A_k(\lambda)\rmd Z(\lambda),\,1\leq k\leq4\right)\right|
  \leq c_4 \;\|A_1\|_2\,\|A_4\|_2\,\|A_2\|_1\|A_3\|_1\;,
\end{equation}
where $c_4$ is a positive constant and
$\|A_k\|_p=(\int_{-\pi}^\pi|A_k(\lambda)|^p\rmd\lambda)^{1/p}$.

\vspace{0.5cm}

\noindent We now give a small series of examples, adapted
from~\cite{moulines:roueff:taqqu:2007:jtsa} to the time varying setting.
\begin{example}[tvFBM($H$)]
  The \emph{Fractional Brownian motion} (FBM) process $\{ B_H(k) \}_{k \in \Z}$
  with Hurst index $H\in(0,1)$ is a discrete-time version of $\{ B_H(t), t \in
  \R \}$, a Gaussian process with mean zero and covariance
\[
\E[ B_H(t) B_H(s)] = \frac{1}{2} \left\{ |t|^{2H} + |s|^{2H} - |t-s|^{2H} \right\} \; .
\]
The spectral density of $\{ \diffop B_H(k) \}_{k \in \Z}$ is then given by
$\lambda\mapsto |1-\rme^{-\rmi\lambda}|^{-2H+1}f_{\mathrm{FBM}}(\lambda;H)$, where
\begin{equation}
\label{eq:smooth:part:FBM}
f_{\mathrm{FBM}}(\lambda;H)=\left|\frac{2 \sin(\lambda/2)}{\lambda}\right|^{2H+1}+
\left|2\sin(\lambda/2)\right|^{2H+1}\sum_{k\neq0}\left|\lambda+2k\pi\right|^{-2H-1}\;.
\end{equation}
The \emph{time varying Fractional Brownian motion} (tvFBM) 
has generalized spectral density~(\ref{eq:genLocSpecDensSemiParam}) with $p=1$, $d(u)=H(u)+1/2\in(1/2,3/2)$
and $f^\ast(u,\lambda)=f_{\mathrm{FBM}}(\lambda;H(u))$,
where $H$ is a Lipschitz mapping of $[0,1]$ into a subset of $(0,1)$. Then~(\ref{eq:genLocSpecDensSemiParamCond}) 
holds with $\beta=(2H(u)+1)\wedge2$. The corresponding non-negative local transfer function is
\begin{equation}
  \label{eq:tvFBMkernel}
A(u,\lambda)=|1- \rme^{-\rmi\lambda}|^{1/2-H(u)}\sqrt{f_{\mathrm{FBM}}(\lambda;H(u))}\;.  
\end{equation}
In this case, by Lemma~\ref{lem:lipshitz_with_log}
in~\ref{sec:technical-lemmas},~(\ref{eq:AssumAtimeSmooth}) holds for any
$D>\sup_u H(u)-1/2$.
\end{example}


\begin{example}[tvFGN($H$)]
The \emph{time varying fractional Gaussian noise} (tvFGN) is defined similarly as the tvFBM by
$f^\ast(u,\lambda)=f_{\mathrm{FBM}}(\lambda;H(u))$ but with $p=0$ and $d(u)=H(u)-1/2\in(-1/2,1/2)$.
\end{example}

\begin{example}[approximated causal tvFBM($H$)] The drawback of the tvFBM (and
  also of tvFGN) defined above is the non-causality of the transfer function
  $A(u,\cdot)$ defined in~(\ref{eq:tvFBMkernel}). Since $\{ \diffop B_H(k)
  \}_{k \in \Z}$ is purely non-deterministic, it admits a causal representation.
  On the other hand, to our knowledge, the corresponding transfer function is
  not explicitly given and thus~(\ref{eq:AssumAtimeSmooth}) is difficult to
  check. To circumvent this problem, one may construct an example by
  approximating 
  a causal continuous time representation of the FBM, see \textit{e.g.}
  \cite[Chapter 6]{samo:2006}. Let us fix $H$ in $(1/2,1)$. 
Replacing the integral by a discrete sum in this
  representation, one obtains the following process
$$
\tilde{B}_H(t)=\sum_{s\in\Z}\{(t-s)_+^{H-1/2}-(-s)_+^{H-1/2}\}\,\varepsilon_s,\quad t\in\Z\;,
$$
where $\{\varepsilon_s\}_{s\in\Z}$ is a standard Gaussian white noise. Then
$$
\diffop \tilde{B}_H(t) =
\sum_{k\geq0}\{k^{H-1/2}-(k-1)_+^{H-1/2}\}\,\varepsilon_{t-k},\quad t\in\Z\;,
$$
is a causal representation of a stationary process. Using an integral
approximation of the discrete Fourier transform of the sequence 
$(k_+^{H-1/2}-(k-1)_+^{H-1/2})_{k\in\Z}$, one can show that,  as $\lambda\to0$,
the corresponding transfer function $A_H(\lambda)$ satisfies
$|A_H(\lambda)|=C_H\,|\lambda|^{1/2-H}+O(1)$ for some positive
constant $C_H$. Moreover, for any $\epsilon>0$, there is a constant $C$ such
that, for all $1/2<H'\leq H<1$ and $\lambda\in(-\pi,\pi)$,
$$
|A_H(\lambda)-A_{H'}(\lambda)|\leq C\,|H-H'|\,|\lambda|^{H-1/2-\epsilon}\;.
$$
Let now $H$ be a Lipschitz mapping of $[0,1]$ into a subset of $(1/2,1)$ and
define the approximated causal tvFBM($H$) process by setting
$A(u,\lambda)=A_{H(u)}(\lambda)$. Then Condition~(\ref{eq:AssumAtimeSmooth})
holds again for any $D>\sup_u H(u)-1/2$.
\end{example}



\begin{example}[tvARFIMA$(q,d,r)$]
  The \emph{time varying autoregressive fractionally integrated moving average}
  (tvARFIMA$(q,d,r)$) process is defined by
\begin{equation}
  \label{eq:arfimaLocSpecDens}
A(u;\lambda)=(1-\rme^{-\rmi\lambda})^{-d(u)+p}\;
\frac{\sigma(u)}{\sqrt{2\pi}}\;
\frac{1+\sum_{k=1}^r\theta_k(u) \rme^{-\rmi\lambda}}
{1 - \sum_{k=1}^q \phi_k(u) \rme^{-\rmi\lambda}}\;,
\end{equation}
where $d:[0,1]\to(-\infty,p+1/2)$, $\sigma:[0,1]\to\R_+$,
$\phi=[\phi_1\,\,\dots\,\,\phi_q]^T:[0,1]\to\R^q$ and
$\theta=[\theta_1\,\,\dots\,\,\theta_r]^T:[0,1]\to\R^r$ are Lipschitz functions
such that $ 1 - \sum_{k=1}^q \phi_k(u) z^k$ does not vanish for all $u\in[0,1]$
and $z\in\C$ such that $|z|\leq1$. Using this latter condition, the local
transfer function $A(u;\cdot)$ defines a causal \emph{autoregressive
  fractionally integrated moving average} (ARFIMA$(q,d(u)-p,r)$) process and
the local generalized spectral density~(\ref{eq:genLocSpecDens}) satisfies the
conditions~(\ref{eq:genLocSpecDensSemiParam})
and~(\ref{eq:genLocSpecDensSemiParamCond}) with $\beta=2$.  Using
Lemma~\ref{lem:lipshitz_with_log} in~\ref{sec:technical-lemmas}, the
Lipschtiz assumptions on $d$, $\sigma$, $\theta$ and $\phi$ yield the
condition~(\ref{eq:AssumAtimeSmooth}) with $D>\sup_u d(u)-p$.
\end{example}

In order to verify Condition~(\ref{eq:AssumAtimeSmooth2}) trivially, the
simplest definition of $\{\diffop^pX_{t,T}\}$ in all the previous examples is
to take $A^0_{t,T}(\lambda)=A(t/T,\lambda)$, that is to set the time-varying
linear representation
\begin{equation}
  \label{eq:LocStatProcSpectral}
\diffop^p X_{t,T} = 
\int_{-\pi}^\pi A(t/T;\lambda)\ \rme^{\rmi\lambda t}\;\rmd Z(\lambda)\;,
\end{equation}
as will be done for our simulated tvARFIMA in Section~5.  However, one might
also want to use a different transfer function $A^0_{t,T}$
in~(\ref{eq:LocStatProcSpectralExt}), provided that
Condition~(\ref{eq:AssumAtimeSmooth2}) holds. Such approximated time varying
linear representation is motivated by the tvAR($p$) process, which satisfies
the recursion
$$
X_{t,T}-\sum_{k=1}^p\phi_k(t/T)X_{t-k,T}=\sigma(t/T)\;\varepsilon_t, \quad 1\leq t\leq T \;,
$$
along with appropriate initial conditions. It has been shown
in~\cite{dahlhaus96} that such non-stationary process does not satisfy a
representation of the form~(\ref{eq:LocStatProcSpectral}) (with $p=0$) but it
does satisfy~(\ref{eq:LocStatProcSpectralExt}) and~(\ref{eq:AssumAtimeSmooth2})
(with $p=D=0$).

\section{Estimation method based on wavelet analysis}\label{sec:3}
\subsection{Discrete wavelet transform (DWT)}
Following the approach presented in \cite{moulines:roueff:taqqu:2007:jtsa} for the estimation of the memory parameter of 
a stationary sequence, we 
compute the discrete wavelet transform (DWT) of $\{X_{t,T},\;1\leq t \leq T\}$ (in discrete time) for a given scale function
$\phi$ and wavelet $\psi$. We denote by $\{W_{j,k;T}\,;j\geq0,k\in\mathbb{Z}\}$ the wavelet coefficients of the
process $\{X_{t,T},\,1\leq t\leq T\}$,
\begin{equation}
  \label{eq:wavcoeffLocStat}
  W_{j,k;T}= \sum_{t=1}^T h_{j,2^jk-t}X_{t,T},\quad k=0,\dots,T_j-1\;,
\end{equation}
where $\{h_{j,t},\;,t\in\Z\}$ denotes the wavelet detail filter at scale $j$ associated to $\phi$ and $\psi$ through the
relation
$$
h_{j,t}=2^{-j/2}\int_{-\infty}^\infty\phi(u+t)\psi(2^{-j}u)\;\rmd u\;, 
$$ 
and $T_j$ the number of available wavelet coefficients at scale $j$, which satisfies 
\begin{equation}
  \label{eq:Tjasymp}
T2^{-j} - c \leq T_j\leq T2^{-j},\quad\text{for some constant $c$ independent of $j\geq 0$}\;.
\end{equation}
The filter $h_{j,\cdot}$ and $T_j$ are
defined so that the support $\{t~:~h_{j,2^jk-t}\neq0\}$ is included in $\{1,\dots T\}$ for $k=0,\dots,T_j-1$.
Observe that here $j$ denotes the scale index of the wavelet coefficient and $k$ its position index. 
We use the convention that a large $j$ corresponds to a coarse scale. Let us define
\begin{equation}
  \label{eq:HjDef}
H_j(\lambda)=\sum_{t\in\Z} h_{j,t}\rme^{-\rmi t\lambda}  
\end{equation}
the corresponding filter transfer function. The following conditions on the wavelet $\psi$ and scale function $\phi$ are
assumed to hold for a positive integer $M$ and a real $\alpha$.
\begin{enum_W}
\item\label{item:Wreg} $\phi$ and $\psi$ are compactly-supported, integrable, $\int_{-\infty}^\infty
  \phi(t)\,\rmd t = 1$ and $\int_{-\infty}^\infty \psi^2(t)\,\rmd t = 1$. 
\item\label{item:psiHat} There exists $\alpha>1$ such that
$\sup_{\xi\in\R}|\hat{\psi}(\xi)|\,(1+|\xi|)^{\alpha} <\infty$, where
$\hat{\psi}(\xi)=\int_{-\infty}^\infty\psi(t)\,\rme^{-\rmi\xi t}\rmd t$ denotes the Fourier transform of $\psi$.
\item\label{item:MVM} The function $\psi$ has  $M$ vanishing moments, $ \int_{-\infty}^\infty t^m \psi(t) \,\rmd t=0$ for all $m=0,\dots,M-1$
\item\label{item:MIM} The function $ \sum_{k\in\Z} k^m\phi(\cdot-k)$
is a polynomial of degree $m$ for all $m=0,\dots,M-1$.
\end{enum_W}
Under~\ref{item:MVM} and~\ref{item:MIM}, the filter can be interpreted as the
convolution of the $\diffop^M$ filter with a finite impulse response
filter. Hence if $M\geq p$, Equation~(\ref{eq:wavcoeffLocStat}) may be written
as
$$
  W_{j,k;T}= \sum_{t=1}^T \tilde{h}_{j,2^jk-t}(\diffop^p X)_{t,T},\quad k=0,\dots,T_j-1\;,
$$
where $h_{j,\cdot}=\tilde{h}_{j,\cdot}\ast\diffop^p$. In particular, we have
\begin{equation}
  \label{eq:HjTilde}
  \tilde{H}_j(\lambda)=\sum_{t\in\Z} \tilde{h}_{j,t}\rme^{-\rmi t\lambda} =
H_j(\lambda)(1-\rme^{\rmi\lambda})^{-p}\;.
\end{equation}

\subsection{Local wavelet spectrum, local scalogram, tangent scalogram, and final estimator}

Recall that $f(u,\cdot)$ in~(\ref{eq:genLocSpecDens}) can be interpreted as a
\emph{local generalized spectral density} at rescaled time $u\in[0,1]$. Hence,
as in the stationary setting used in \cite{moulines:roueff:taqqu:2007:jtsa},
for each such $u$, we may define a \emph{local wavelet spectrum}
$\sigma^2(u)=\{\sigma^2_j(u),\,j\geq0\}$, where for each $j\geq0$,
$\sigma^2_j(u)$ is the variance of the wavelet coefficients at scale index $j$
of a process with generalized spectral density $f(u,\cdot)$. This variance is
well defined under the assumption $M\geq p$ because in this case the wavelet
coefficients at given scale are weakly stationary. Moreover,
by~(\ref{eq:genLocSpecDens}) and~(\ref{eq:HjTilde}),
\begin{equation*}
\sigma^2_j(u)  = \int_{-\pi}^\pi \left|{H}_j(\lambda)\right|^2\,f(u;\lambda)\;\rmd \lambda 
= \int_{-\pi}^\pi \left|\tilde{H}_j(\lambda)\ A(u;\lambda)\right|^2\;\rmd \lambda \;.
\end{equation*}
The following intuitive definition will be also useful when developing our
estimation theory using stationary approximations.  For any $u\in [0,1]$ one
may define a \emph{tangent} stationary process for the $p$-th increment
\begin{equation}
  \label{eq:tangentProcSpectral}
\diffop^pX_{t}(u) = \int_{-\pi}^\pi A(u;\lambda)\ \rme^{\rmi\lambda t}\;\rmd Z(\lambda)\ , 
\end{equation}
whose spectral density is $|1- \rme^{-\rmi\lambda}|^{2p}f(u,\lambda)$. 
Further we define the wavelet coefficients of the tangent weakly stationary
process at any $u\in [0,1]$, namely,
\begin{align}
  \label{eq:wavcoeffStat}
  W_{j,k}(u) &= \sum_{t=1}^T \tilde{h}_{j,2^jk-t}(\diffop^p X)_{t}(u)\\
  \label{eq:wavcoeffStatSpectral}
&=\int_{-\pi}^\pi \tilde{H}_j(\lambda)\ A(u;\lambda)\ \rme^{\rmi\lambda 2^jk}\;\rmd Z(\lambda)  \;,
\quad k=0,\dots,T_j-1\;.
\end{align}
 these wavelet coefficients are indeed those of a process with generalized
spectral density $f(u,\cdot)$. Thus the above definition gives
\begin{equation}
  \label{eq:defLocWavSpec}
\sigma^2_j(u)= \E\left[W_{j,k}^2(u)\right]  \;.
\end{equation}

\vskip0.3cm

An important tool for the estimation of the long memory is the \emph{scalogram} (first introduced in this context
by~\cite{wornell:oppenheim:1992} and developed by~\cite{abry:veitch:1998}) defined as
$$
\widehat{\sigma}^2_{j}=T_j^{-1}\sum_{k=0}^{T_j-1}W_{j,k}^2 \; .
$$
Here to cope with local stationarity, we will need a \emph{local scalogram} for estimating the local wavelet spectrum,
namely, for a given $u\in [0,1]$,
\begin{equation}
  \label{eq:EstimatorWeights}
\widehat{\sigma}^2_{j,T}(u)= \sum_{k=0}^{T_j-1}\gamma_{j,T}(k)W_{j,k;T}^2 \; ,
\end{equation}
where $\{\gamma_{j,T}(k)\}$ are some non-negative weights localized at indices $k\approx uT_j$ and normalized in 
such a way that   
\begin{equation}
    \label{eq:normWeights}
\sum_{k=0}^{T_j-1}\gamma_{j,T}(k) =1 \; .
\end{equation}
The localization property will be expressed by imposing a bound on the increase rate of the following quantity (see
equation~(\ref{eq:NormalizedWeightsSums})) 
\begin{equation}
  \label{eq:GammaDefQ}
\Gamma_q(u;j,T)=\sum_{k=0}^{T_j-1}|\gamma_{j,T}(k)|\;\;|k-Tu2^{-j}|^q  \;,
\end{equation}
as $T\to\infty$ for appropriate values of the exponent $q$.

An important tool for studying the local scalogram is the \emph{tangent scalogram} defined as
\begin{equation}
  \label{eq:tangentScalogram}
\widetilde{\sigma}^2_{j,T}(u)=\sum_{k=0}^{T_j-1}\gamma_{j,T}(k)W_{j,k}^2(u)\;.  
\end{equation}
We note that this definition is similar to that of the local
scalogram in~(\ref{eq:EstimatorWeights}) but with the wavelet coefficients $W_{j,k;T}$ replaced by the tangent wavelet
coefficients $W_{j,k}^2(u)$ defined in~(\ref{eq:wavcoeffStat}).
The tangent scalogram is not an estimator since it cannot be computed from the observations $X_{1,T},\dots,X_{T,T}$.
However, it provides useful approximations of the local scalogram. 

We conclude this section by deriving an {\em estimator of the time-varying long memory parameter}.
The local wavelet spectrum is related to the local memory parameter $d(u)$ by the asymptotic property $\sigma^2_j(u)\sim c
2^{2d(u)j}$ as $j\to\infty$. This property will be made more precise when we study the bias, see the
relation~(\ref{eq:BiasControl}) below. 
An estimator of $d(u)$ is obtained by a linear regression of  $(\log \widehat{\sigma}^2_{j,T}(u))_{j=L,\dots,L+\ell}$ with
respect to $j=L,\dots,L+\ell$, where $\ell$ is the number of scales used in the regression and $L$ is the lowest
scale index used in the regression. Let $\bw$ be a vector $\bw =[w_0,\dots,w_{\ell}]^T$  satisfying 
\begin{equation}
\label{eq:propertyw}
\sum_{i=0}^{\ell} w_{i}  = 0\quad\text{and}\quad 2 \log(2) \sum_{i=0}^{\ell} i w_{i}  = 1 \;.
\end{equation}
The local estimator of $d(u)$ is defined as 
\begin{equation}
\label{eq:definition:estimator:regression}
\hat{d}_T(L) = \sum_{j=L}^{L+\ell} w_{j-L} \log \left( \widehat{\sigma}^2_{j,T}(u)\right) \; .
\end{equation}

\subsection{Conditions on the weights $\gamma_{j,T}(k)$ and examples}
\label{sec:exampl-weights}

Let us now precise our conditions on the weights $\gamma_{j,T}(k)$.
Denote, for any $0\leq i\leq j$, $v\in\{0,\dots,2^{i}-1\}$ and $\lambda\in\R$,
\begin{equation}
  \label{eq:defweightsFourier}
\Phi_{j,T}(\lambda;i,v)=\sum_{l\in \calT_j(i,v)}\gamma_{j-i,T}(2^i l+v)\rme^{\rmi l \lambda}\;,
\end{equation}
where
\begin{equation}
  \label{eq:calT}
   \calT_j(i,v)= \left\{l \,:\, 0\leq l < 2^{-i}(T_{j-i}-v)\right\}\;.
\end{equation}
We moreover define
\begin{equation}\label{eq:deltaDef}
\delta_{j,T}=\sup_{k=0,\dots,T_j-1}|\gamma_{j,T}(k)| \; .
\end{equation}

The weights
$\gamma_{j,T}(k)$ must satisfy an appropriate asymptotic behavior as $T\to\infty$ for obtaining a
consistent estimator of $d(u)$. More precisely, the following assumption will be required.      

\begin{assumption}\label{assump:weights}
The index $j$ depends on $T$ so that the weights $(\gamma_{j,T}(k))_{k}$ satisfy the following asymptotic properties as
$T\to\infty$. 
  \begin{enumerate}[(i)]
  \item\label{item:delatAssump} We have $\delta_{j,T}\to0$, and for any fixed integer $i$,
    $\delta_{j+i,T}\sim2^{i}\delta_{j,T}$. 
  \item\label{item:Vassump} For all $i,i'\geq0$, $v\in\{0,\dots,2^{i}-1\}$ and $v'\in\{0,\dots,2^{i'}-1\}$, there exists a constant
    $V=V(i,v;i',v')$ such that
\begin{equation}\label{eq:VasympDef}
\delta_{j,T}^{-1}\int_{-\pi}^\pi \Phi_{j,T}(\lambda;i,v)\overline{\Phi_{j,T}(\lambda;i',v')}\,\rmd\lambda \to V(i,v;i',v')\;.
\end{equation}
\item\label{item:neglicAssump}  For all $\eta>0$, $i\geq0$ and $v\in\{0,\dots,2^{i}-1\}$, we have
$$
\delta_{j,T}^{-1/2} \sup_{\eta\leq|\lambda|\leq\pi}\left|\Phi_{j,T}(\lambda;i,v)\right|\to 0 \; .
$$
\item\label{item:Locassump} For $q=0,1,2$, we have
\begin{equation}
  \label{eq:NormalizedWeightsSums}
  \Gamma_q(u;j,T)=O\left((\delta_{j,T})^{-q}\right)\;,
\end{equation}
where $\Gamma_q(u;j,T)$ is defined in~(\ref{eq:GammaDefQ}).
  \end{enumerate} 
\end{assumption}

We provide several examples of weights satisfying this assumption below.
In these examples, the weights $\gamma_{j,T}(k)$, $k=0,\ldots,T_j$, are entirely determined by $T_j$ and a bandwidth
parameter $\band_T$ and
\begin{equation}
  \label{eq:bandwidthDef}
\delta_{j,T}^{-1}\asymp \band_T T_j \sim \band_T T2^{-j} \;.  
\end{equation}
In kernel estimation, one may interpret the bandwidth parameter $\band_T$ as 
the proportion of wavelet coefficients used for the estimation of the local scalogram
$\widehat{\sigma}^2_{j,T}(u)$ at given scale  $j$ and position $u$, among the $T_j$  wavelet coefficients available at scale
$j$ from $T$ observations $X_{1,T},\dots,X_{T,T}$.   
Lemmas~\ref{lem:two-sided-kernel} and~\ref{lem:recursive-weights} show that, for these examples, 
Assumption~\ref{assump:weights} is satisfied as soon as $T_j\to\infty$ and $\band_T T_j\to0$, 
except in the non-compactly supported case (K-\ref{item:exp}) in  Lemma~\ref{lem:two-sided-kernel} where we assume 
in addition that  $T_j\exp(-c' \band_{T}T_j)=O(1)$ for any $c'>0$, which holds in the typical asymptotic setting  
$\band_T \asymp T_j^{-\zeta}$ with $\zeta\in(0,1)$. 

\begin{example}[Two-sided kernel weights]
\label{ex:two-sided-kernel}
For  $u\in(0,1)$, one has a number of observations before rescaled time $u$ and
after rescaled time $u$ both tending to infinity.   
Thus we may use a two-sided kernel to localize the memory parameter estimator around $u$.  
For a given bandwidth $b=b_T$, the corresponding weights are given by 
\begin{equation}
  \label{eq:KernelWeights}
\gamma_{j,T}(k)=\rho_{j,T}^{-1}\;K( (uT_j-k) / (\band_{T}T_j) )\; ,
\end{equation}
where $K$ is a non-negative symmetric function and $\rho_{j,T}$
is a normalizing term so that~(\ref{eq:normWeights}) holds. In the last display we see that $\band_T$ is the bandwidth 
in a rescaled time sense while $\band_{T}T_j$ is the corresponding  bandwidth in the sense of location indices
$k=0,1,2,\dots,T_j$ at scale $j$ .
Lemma~\ref{lem:two-sided-kernel} in the appendix shows that Assumption~\ref{assump:weights} holds
for a wide variety of choices for $K$. In particular
Assumption~\ref{assump:weights} holds  with
$\delta_{j,T}\asymp(b_TT_j)^{-1}$ and $V(i,v,;i',v')=2\pi2^{-i-i'}$ under the
following assumption. 
\begin{assumption}\label{assump:rect_weights}
The weights $(\gamma_{j,T}(k))$ are defined by~(\ref{eq:KernelWeights}) with
$K=\1_{[-1/2,1/2]}$. Moreover, as $T\to\infty$, $b_T\to0$ and $T_jb_T\to\infty$.
\end{assumption}
\end{example}

\begin{example}[Recursive weights]
\label{ex:recursive-kernel}
By \emph{recursive weights}, we here mean that, given $T,\;L$ and $\bw$,  the possibility of computing
$\widehat{\sigma}^2_{j,T}(u)$ by successive simple linear processing involving a finite number of operations 
after each new observations $X_{t;T}$ as $t$ grows from $t=1$ to $t=T$.

Because the DWT is defined as a succession of finite filtering and decimation, it is indeed possible to compute $W_{j,k;T}$
online for all $j\in\{L,\dots,L+\ell\}$ and $k\in\{0, \dots,T_j\}$. Then an online implementation of the local \emph{recursive} scalogram can be done by setting  
$$
\widehat{\sigma}^2_{j,-1;T}=0\;, \quad j\in\{L,\dots,L+\ell\},
$$ 
and, using the following recursive equation for all $j\in\{L,\dots,L+\ell\}$ and $t\in\{0,\dots,T_j-1\}$,   
$$
\widehat{\sigma}^2_{j,t;T}=\exp(-(\band_{T}T_j)^{-1}) \; \widehat{\sigma}^2_{j,t-1;T} + W_{j,t;T}^2\;,
$$
where $(\band_{T}T_j)^{-1}$ is the exponential \emph{forgetting exponent} corresponding to  
the \emph{bandwidth parameter} $\band_{T}$. For any $u\in (0,1]$, we define a local recursive scalogram by  
$$
\widehat{\sigma}^2_{j;T}(u)=\rho_{j,T}^{-1} \widehat{\sigma}^2_{j,[u\,T_j]-1;T}\; ,
$$
where $[a]$ denotes the integer part of $a$ and 
\begin{equation}
\label{eq:rec_rho}
\rho_{j,T}=\sum_{k=0}^{[u\,T_j]-1} \rme^{-k/(\band_{T}T_j)}
= \frac{1-\rme^{-[u\,T_j]/(\band_{T}T_j)}}{1-\rme^{-(\band_{T}T_j)^{-1}}} \;.
\end{equation}
Hence~(\ref{eq:EstimatorWeights}) and~(\ref{eq:normWeights}) hold with   
\begin{equation}
  \label{eq:recWeights}
  \gamma_{j,T}(k)= \rho_{j,T}^{-1} \rme^{-([u\,T_j]-1-k)/(\band_{T}T_j)} \1_{[0,uT_j-1)}(k) \;. 
\end{equation}
Observe that these weights are one-sided by construction, since only the observations before rescaled time $u$ are used for estimating
$d(u)$. This is the reason why we may take $u\in (0,1]$.  
Lemma~\ref{lem:recursive-weights} shows that Assumption~\ref{assump:weights} holds for these weights, provided that
$\band_{T}\to0$ and $T_j\band_{T}\to\infty$. 
\end{example}

\section{Asymptotic properties}\label{sec:4}

We study the asymptotic properties of $\widehat{d}_T(L)$ defined by~(\ref{eq:definition:estimator:regression}) as 
$L,T\to\infty$ in such a way that Assumption~\ref{assump:weights} holds for each $j=L,L+1,\dots,L+\ell$ and for the chosen
weights $\gamma_{j,T}(k)$. We provide further conditions on $L,T,\delta_{L,T}$ under which consistency holds and derive
the corresponding rate of convergence (Corollary~\ref{cor:rate-param-est}). Under strengthened conditions, we further show that $\widehat{d}_T(L)$ is
asymptotically normal (Theorem~\ref{theo:clt-param-est}). These results essentially follow from asymptotic results on the tangent
scalogram (Theorem~\ref{thm:clt-tang-scal}, Relations~(\ref{eq:BiasControl}) and~(\ref{eq:VarTildeBound})) and approximation
results on the local scalogram (Proposition~\ref{prop:local-scalogram}) based on the tangent scalogram.

\subsection{Asymptotic properties of the local scalogram}
In order to derive asymptotic results for $\widehat{\sigma}^2_{j,T}(u)$, we
first establish a bound on the error made when approximating
$\widehat{\sigma}^2_{j,T}(u)$ by $\widetilde{\sigma}^2_{j,T}(u)$.

\begin{prop}\label{prop:local-scalogram}
  Let $u\in[0,1]$ and consider a model satisfying
  Assumption~\ref{ass:model_lin}. Assume~\ref{item:Wreg}--\ref{item:MIM} hold with $M\geq
  p\vee (d(u)-1/2)$ and $\alpha>1/2-d(u)$.  Suppose moreover that
  Assumption~\ref{assump:weights}(\ref{item:Locassump}) hold.  Then, the
  following approximation holds.
\begin{equation}
  \label{eq:ScalApproxResult}
    \E\left[\left(\widehat{\sigma}^2_{j,T}(u)-\widetilde{\sigma}^2_{j,T}(u)\right)^2\right]=
 O\left(2^{(6+4p)j}T^{-4} \delta_{j,T}^{-4} + 2^{(3+2p+2d(u))j}T^{-2} \delta_{j,T}^{-2}\right) \; .
\end{equation}
\end{prop}

Next, we derive a bound of the mean square error for estimating
$f^*(u,0)\constantePsiD(d(u))\;2^{2jd(u)}$ using the estimator
$\widehat{\sigma}^2_{j,T}(u)$, where $\constantePsiD$ is the function defined
by
\begin{equation}
  \label{eq:Kdef}
  \constantePsiD(d)= \int_{-\infty}^{\infty}|\xi|^{-2d}\,|\hat\psi(\xi)|^2\;\rmd\xi\;,\quad   1/2-\alpha<d< M+1/2\;.
\end{equation}
In fact as the estimator $\hat{d}_T(L)$ is defined
in~(\ref{eq:definition:estimator:regression}) using
$\widehat{\sigma}^2_{j,T}(u)$ with $j=L+i$, $i=0,\dots,\ell$, and as
$L,T\to\infty$, it will be convenient to normalize these quantities by
$2^{2Ld(u)}$, so that
$f^*(u,0)\constantePsiD(d(u))\;2^{2jd(u)}/2^{2Ld(u)}=f^*(u,0)\constantePsiD(d(u))\;2^{2id(u)}$
does not depend on $L$.

\begin{theo}\label{theo:rate-param-scalo}
  Let $u\in[0,1]$ and consider a model satisfying
  Assumption~\ref{ass:model_lin}. Assume~\ref{item:Wreg}--\ref{item:MIM} hold
  with $M\geq p\vee d(u)$ and $\alpha>(1+\beta)/2 -d(u)$. Then we have, as
  $j\to\infty$,
\begin{equation}
  \label{eq:BiasControl}
\sigma_j^2(u)=
f^*(u,0)\constantePsiD(d(u))\;2^{2jd(u)}\left\{1+O\left(2^{-\beta j}\right)\right\}\;.
\end{equation}
Suppose moreover that Assumption~\ref{assump:weights} holds and that
  \begin{equation}
    \label{eq:condConsistency}
    2^{(3+2\{p-d(u)\})L}T^{-2} \delta_{L,T}^{-2}\to0\;.
  \end{equation}
Then we have for $j=L+i$ with $i=0,\dots,\ell$,
\begin{multline}
\label{eq:ratesOp}
\E\left[(2^{-2Ld(u)}\widehat{\sigma}^2_{j,T}(u)-f^*(u,0)\constantePsiD(d(u))\;2^{2id(u)})^2\right]\\
=O\left(\delta_{L,T}+2^{(3+2\{p-d(u)\})L}T^{-2} \delta_{L,T}^{-2}+2^{-2\beta L}\right)
\end{multline}
\end{theo}

Using the approximation result stated in
Proposition~\ref{prop:local-scalogram}, we may also wish to obtain a central
limit theorem (CLT) for the local scalogram. To this end, we must first derive
a CLT for the tangent scalogram.  Define, for any integer $\ell\geq0$ and
$d\in(1/2-\alpha,M]$ the $2^\ell$-dimensional cross spectral density
$\bD_{\infty,\ell}(\lambda;d)=[\bD_{\infty,\ell,v}(d)]_{v=0,\dots,2^\ell-1}$ of
the DWT of the generalized fractional Brownian motion (see
\cite{moulines:roueff:taqqu:2007:jtsa}) by
$$
\bD_{\infty,\ell}(\lambda;d)= \sum_{l\in\Z} |\lambda+2l\pi|^{-2d}\,\be_{\ell}(\lambda+2l\pi) \,
\overline{\hat{\psi}(\lambda+2l\pi)}\hat{\psi}(2^{-\ell}(\lambda+2l\pi))\;,
$$
where for all $\xi\in\R$, 
$$ 
\be_\ell(\xi) = 2^{-\ell/2}\, [1, \rme^{-\rmi2^{-\ell}\xi}, \dots, \rme^{-\rmi(2^{\ell}-1)2^{-\ell}\xi}]^T\;.
$$
In other words $\bD_{\infty,\ell}(\lambda;d)$ is a vector with entries
$$
\bD_{\infty,\ell,v}(\lambda;d)= 2^{-\ell/2}\, \sum_{l\in\Z} |\lambda+2l\pi|^{-2d}\,
\rme^{-\rmi \,v\,2^{-\ell}(\lambda+2l\pi)}
 \,
\overline{\hat{\psi}(\lambda+2l\pi)}\hat{\psi}(2^{-\ell}(\lambda+2l\pi)),\quad 0\leq v < 2^\ell\;.
$$
We can now state the CLT for the tangent scalogram.  
\begin{theo}\label{thm:clt-tang-scal}
  Let $u\in[0,1]$ and consider a model satisfying
  Assumption~\ref{ass:model_lin}. Suppose that~\ref{item:Wreg}--\ref{item:MIM}
  hold with $M\geq p\vee d(u)$, $\alpha>1/2-d(u)$. Suppose moreover one of the
  two following assertion holds.
  \begin{enumerate}[(a)]
  \item Assumption~\ref{assump:rect_weights} holds and $\{\varepsilon_t\}$ is  an i.i.d. sequence.
  \item
  Assumptions~\ref{assump:weights}~(\ref{item:delatAssump})--(\ref{item:neglicAssump}) hold
  and $\{\varepsilon_t\}$ is a Gaussian process.
  \end{enumerate}
Then, for any $\ell\geq0$, the following weak convergence holds.
\begin{equation}\label{eq:S_jCLT}
\left(\widetilde{S}_L(u)-\E\left[\widetilde{S}_L(u)\right]\right)\Rightarrow \calN(0,(f^*(u,0))^2\Sigma(u)) \;,
\end{equation}
where
\begin{equation}\label{eq:S_jdef}
\widetilde{S}_L(u)=2^{-2Ld(u)}\delta_{L,T}^{-1/2}
[\widetilde{\sigma}^2_{L,T}(u)\,\,\widetilde{\sigma}^2_{L+1,T}(u)\,\,\dots\,\,\widetilde{\sigma}^2_{L+\ell,T}(u)]^T \; ,
\end{equation}
and $\Sigma(u)$ is the $(\ell+1)\times(\ell+1)$ symmetric matrix defined by
\begin{equation}\label{eq:Sigmadef}
\Sigma_{i,i'}(u)=2\ 2^{\{1+4d(u)\}i}\sum_{v=0}^{2^{i-i'}-1}V(0,0;i-i',v)
\int_{-\pi}^\pi\left|\bD_{\infty,i-i',v}(\lambda;d(u))\right|^2\,\rmd\lambda\;,
\end{equation}
on the bottom-left triangle $0\leq i'\leq i\leq \ell$ with
$V(0,0;i-i',v)$ defined in~(\ref{eq:VasympDef}). 
\end{theo}
\begin{remark}
  A CLT for the sum of squares of the wavelet coefficients of a stationary long
  memory process was established in~\cite{moulines-roueff-taqqu-2008} for
  Gaussian processes and extended in~\cite{roueff-taqqu-2009b} for linear
  processes.  We separate two sets of assumptions in
  Theorem~\ref{thm:clt-tang-scal}. The result in the linear case is directly
  applicable under Assumption~\ref{assump:rect_weights} in (a) since the
  weights are constant.  To obtain a CLT for general weights
  (Assumption~\ref{assump:weights} in (b)) we use the additional Gaussian
  assumption.  To avoid the Gaussian assumption for such general weights, one
  needs to revisit the results in~\cite{roueff-taqqu-2009} to obtain a CLT for
  sums of weighted squares of decimated linear processes. Such an extension
  goes beyond the scope of this article.
\end{remark}
Applying Proposition~\ref{prop:local-scalogram} and Theorem~\ref{thm:clt-tang-scal}, we immediately get the following result. 
\begin{cor}\label{cor:clt-local-scalogram}
  Let both the assumptions of  Proposition~\ref{prop:local-scalogram} and
  Theorem~\ref{thm:clt-tang-scal} hold. Let $L$ be
  such that
  \begin{equation}
    \label{eq:approxScalNeglict}
2^{(3+2\{p-d(u)\})L}T^{-2} \delta_{L,T}^{-3}\to0\;.    
  \end{equation}
Then, for any $\ell\geq0$, the following weak convergence holds.
\begin{equation}\label{eq:hatS_jCLT}
\left(\widehat{S}_L(u)-\E\left[\widehat{S}_L(u)\right]\right)\Rightarrow \calN(0,(f^*(u,0))^2\Sigma(u)) \;,
\end{equation}
where
\begin{equation}\label{eq:hatS_jdef}
\widehat{S}_L(u)=2^{-2Ld(u)}\delta_{L,T}^{-1/2}
[\widehat{\sigma}^2_{L,T}(u)\,\,\widehat{\sigma}^2_{L+1,T}(u)\,\,\dots\,\,\widehat{\sigma}^2_{L+\ell,T}(u)]^T \; .
\end{equation}
and $\Sigma(u)$ is the $(\ell+1)\times(\ell+1)$ symmetric matrix defined by~(\ref{eq:Sigmadef}).
\end{cor}

\subsection{Asymptotic properties of the estimator $\widehat{d}_T(L)$}
The following result establishes the consistency of the estimator $\widehat{d}_T(L)$ defined
in~(\ref{eq:definition:estimator:regression}) with $\bw =[w_0,\dots,w_{\ell}]^T$ fulfilling~(\ref{eq:propertyw}) and provides a rate of convergence.
\begin{cor}\label{cor:rate-param-est}
Under the same assumptions as Theorem~\ref{theo:rate-param-scalo}, if moreover $L\to\infty$, then we have
\begin{equation}
\label{eq:ratesOphatd}
\widehat{d}_T(L)=d(u)+O_p\left(\delta_{L,T}^{1/2}+2^{(3/2+\{p-d(u)\})L}T^{-1} \delta_{L,T}^{-1}+2^{-\beta L}\right)=d(u)+o_p(1)\;.
\end{equation}
\end{cor}

Let us determine the optimal rate of convergence of $\widehat{d}_T(L)$ towards $d(u)$. 
By balancing the three terms in the right-hand side of~(\ref{eq:ratesOphatd}), we find that for $2^L\asymp
T^{2/(3+6\beta-2d(u)+2p)}$ and $\band_T\asymp (T_L\delta_{L,T})^{-1}\asymp 
T^{(2d(u)-2p-2\beta-1)/(3+6\beta-2d(u)+2p)}$, these three terms are asymptotically of the same order.
Hence for this choice of the lowest scale $L$ and the bandwidth $\band_T$ (recall that $\delta_{L,T}^{-1}\asymp \band_T
T2^{-L}\to\infty$), we get 
$$
\widehat{d}_T(L)=d(u)+O_p\left(T^{-2\beta/(3+6\beta+2\{p-d(u)\})}\right) \; .
$$
We observe that the rate of convergence depends on the unknown parameter
$d(u)$.  The dependence in $d(u)$ comes from the approximation
result~(\ref{eq:ScalApproxResult}), which appears in~(\ref{eq:ratesOphatd}) in
the term $2^{(3/2+\{p-d(u)\})L}T^{-1} \delta_{L,T}^{-1}$.  Other error terms
in~(\ref{eq:ratesOphatd}) have rates not depending on $d(u)$, which is
consistent with the facts that 1) the rate of convergence does not depend on
$d$ in the stationary case \cite[Theorem~2]{moulines:roueff:taqqu:2007:jtsa},
and 2) these two terms come from the tangent weakly stationary approximation.
On the other hand, the term $2^{(3/2+\{p-d(u)\})L}T^{-1} \delta_{L,T}^{-1}$ may
seem unusual for estimating the time-varying parameter for local-stationary
processes. For instance, for a time-varying AR (tvAR) process with a
Lipschitz-continuous parameter corresponding to~(\ref{eq:AssumAtimeSmooth2})
with $D=0$, the approximation error due to non-stationarity yields the error
term $\band_T\asymp (T\delta_{L,T})^{-1}$.  Indeed this corresponds to the term
$(n\mu)^{-\beta}$ with $\beta=1$ in
\cite[Theorem~2]{moulines-priouret-roueff-2005} which is shown to yield a rate
optimal convergence in Theorem~4 of the same reference.  Our error term is
always larger as it includes the additional multiplicative term
$2^{(3/2+\{p-d(u)\})L}$ which tends to $\infty$ since $p-d(u)>-1/2$ and
$L\to\infty$. Although we cannot assert that our rate is optimal, it can be
explained as follows. In contrast to the tvAR process, our setting is locally
semi-parametric, which implies to let $L$ tend to $\infty$ in order to capture
the low frequency behavior driven by the memory parameter $d$.  It is thus
reassuring that if $L$ were allowed to remain fix our error bound would be of
the same order as for the locally parametric setting. The fact that letting
$L\to\infty$ decreases the rate of convergence is not surprising as the low
frequency behavior implies large lags in the process, which naturally worsens
the quality of the local stationary approximations.  To conclude this
discussion, it is interesting to note that the wavelet estimation of the memory
parameter of a non-linear process may also yield a rate of convergence
depending on the unknown parameter. It is indeed the case for the
infinite-source Poisson process, see
\cite[Remark~4.2]{fay:roueff:soulier:2007}.

We now state the asymptotic normality of the estimator, which mainly follows by
applying Proposition~\ref{prop:local-scalogram}, Theorem~\ref{thm:clt-tang-scal}, the bound~(\ref{eq:BiasControl}) and
the $\delta$--method as in~\cite{moulines:roueff:taqqu:2007:fractals}. 

\begin{theo}\label{theo:clt-param-est}
  Let the assumptions of Corollary~\ref{cor:clt-local-scalogram} hold with
  $\alpha>(1+\beta)/2-d(u)$. Let $L$ be  such that
  \begin{equation}
    \label{eq:approxScalandBiasNeglict}
2^{(3+2\{p-d(u)\})L}T^{-2} \delta_{L,T}^{-3}\to0\quad\text{and}\quad
2^{-2\beta L}\delta_{L,T}^{-1}\to0 \;.    
  \end{equation}
Then, the following weak convergence holds:
\begin{equation}
  \label{eq:CLTd}
\delta_{L,T}^{-1/2}(\widehat{d}_T(L)-d(u))\Rightarrow \calN(0,\mathcal{V}(u))\;,
\end{equation}
where $\widehat{d}_T(L)$ is defined by~(\ref{eq:definition:estimator:regression}) and
$$
\mathcal{V}(u) = \frac1{\constantePsiD^2(d(u))}\sum_{i,i'=0}^{\ell}\Sigma_{i,i'}(u)2^{-2(i+i')d(u)}w_iw_{i'}\;.
$$
with $\Sigma(u)$ and  $\constantePsiD(d(u))$ defined by~(\ref{eq:Sigmadef}) and~(\ref{eq:Kdef}), respectively.
\end{theo}

\section{Numerical examples}\label{sec:5}

We used a Daubechies wavelet with $M=2$ vanishing moments and Fourier decay $\alpha=1.34$ (see \cite{fay-mouline-roueff-taqqu-2009}).
Hence our asymptotic results hold for $-0.84<(1+\beta)/2-\alpha< d(u)\leq M=2$ (the left bound $-0.84$ corresponds to choose $\beta$
arbitrarily small). In particular $d(u)$ will be allowed to take values beyond the unit root case ($d(u)\geq 1$).

\subsection{Simulated data}

We simulate a   $T=2^{12}$-long sample  $X_{1,T},\dots,X_{T,T}$ of a tvARFIMA(1,$d$,0) process  
which has a local spectral density given by~(\ref{eq:arfimaLocSpecDens}) with $\sigma\equiv1$, $\phi_1\equiv0.8$ and 
$$
d(u)=(1-\cos(\pi u/2))/3,\quad u\in[0,1]\;. 
$$
The obtained simulated data is represented in Figure~\ref{fig:raw_simu}.
\begin{figure}[htbp]
\centering
\includegraphics[width=0.6\textwidth]{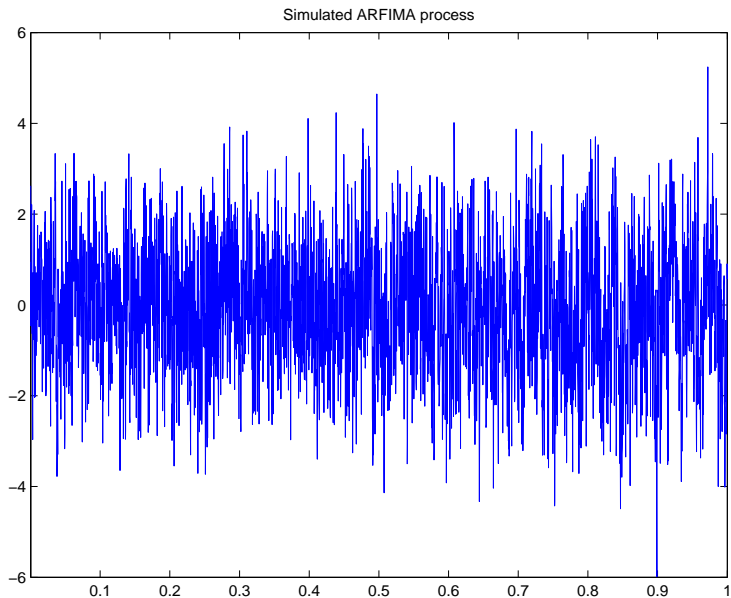}
\caption{A simulated tvARFIMA(1,$d$,0) of length $T=2^{12}$.}
\label{fig:raw_simu}
\end{figure}
We compute the local estimator $\widehat{\sigma}^2_{j,T}(u)$ defined
in~(\ref{eq:EstimatorWeights}) with $\{\gamma_{j,T}(k)\}$ given by the kernel
weights on the one hand and the recursive weights on the other hand, for
$j=1,2,\dots,5$ with a bandwidth $\band_T=0.25$. For the kernel weight we took
the rectangle kernel $K=\1_{[-1/2,1/2]}$.  The obtained local
scalograms $\widehat{\sigma}^2_{j,T}(u)$ of the local wavelet spectrum
$\sigma^2_j(u)$, $j=1,2,\dots,5$, $u\in[0,1]$ are represented in the lower
parts of Figures~\ref{fig:estim_simu_two_sided} and~\ref{fig:estim_simu_rec},
respectively, with a $y$-axis in a logarithmic scale.
\begin{figure}[htbp]
\centering
\includegraphics[width=0.8\textwidth]{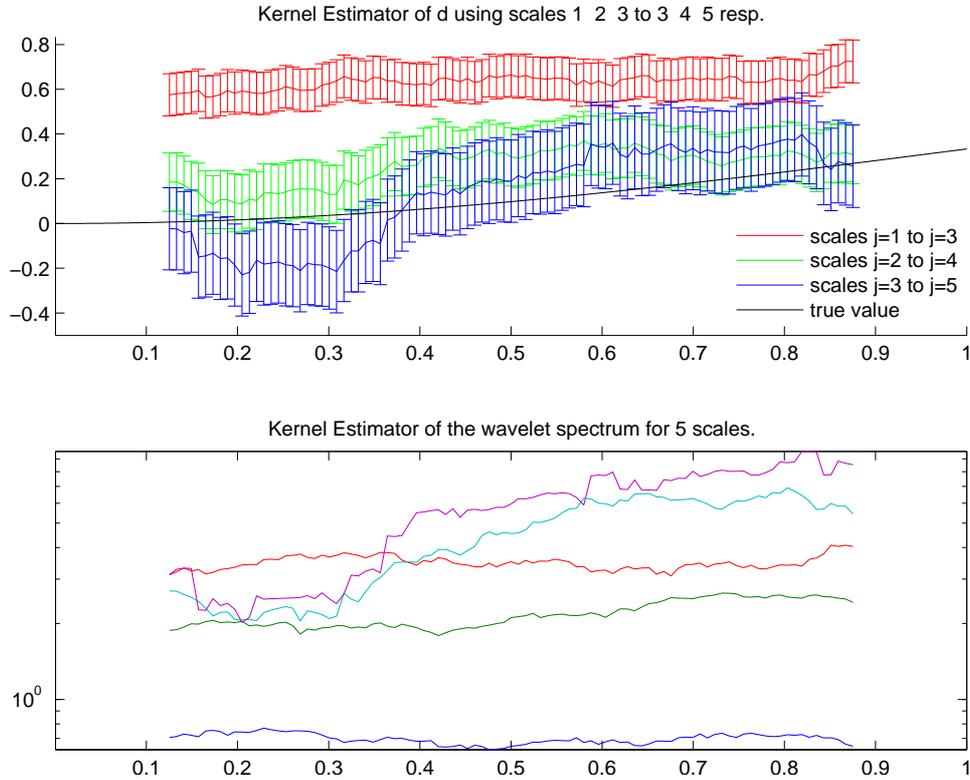}
\caption{Local estimates as functions of $u\in[0,1]$ for the simulated tvARFIMA(1,$d$,0)  using a two-sided rectangular
  kernel. Top: $\hat{d}_T(L;u)$ using scales $j=1,2,3$ to $3,4,5$ (respectively in blue, green and red) and the true value
  $d(u)$ (in thin black). Bottom: $\widehat{\sigma}^2_{j,T}(u)$ for $j=1,\dots,5$.}
\label{fig:estim_simu_two_sided}
\end{figure}
\begin{figure}[htbp]
\centering
\includegraphics[width=0.8\textwidth]{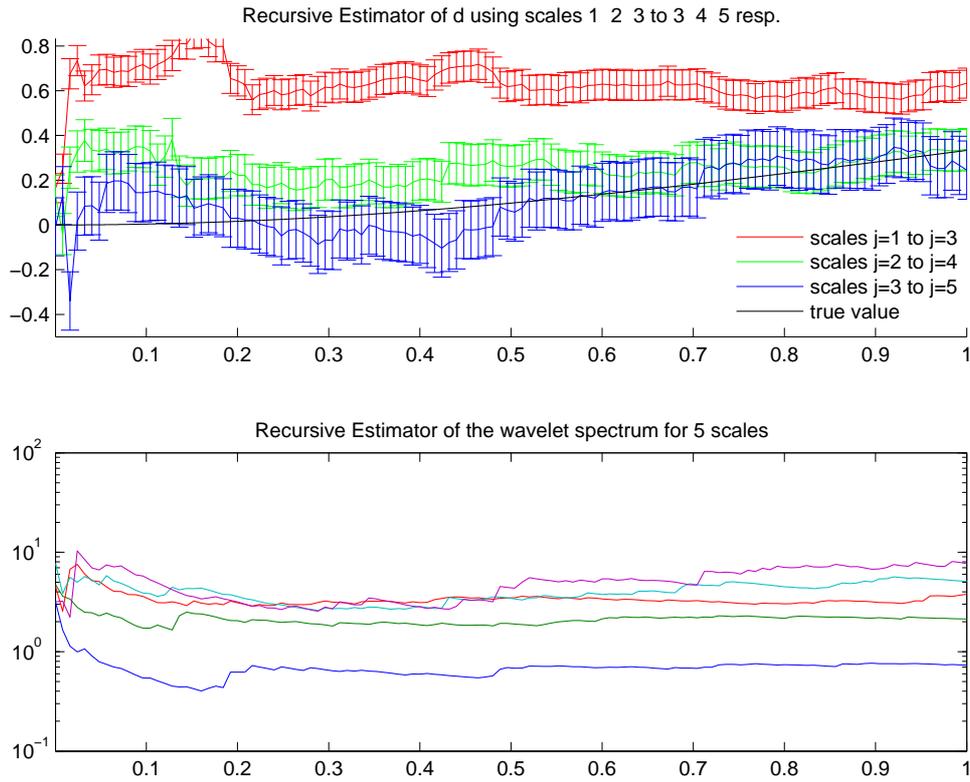}
\caption{Same as Figure~\ref{fig:estim_simu_two_sided} using a recursive estimator.}
\label{fig:estim_simu_rec}
\end{figure}
The five corresponding curves exhibit different variabilities, the larger $j$,
the larger the variability, which is in accordance with our theoretical
findings.  On the top of these two figures, we represented the true parameter
$d(u)$, $u\in[0,1]$ (plain black) and the corresponding estimators
$\widehat{d}_T(u)$ for three sets of scales, namely $j=1,2,3$ (blue line),
$j=2,3,4$ (green line) and $j=3,4,5$ (red line), which correspond to $L=1,2,3$,
respectively, and $\ell=2$ in the three sets of scales.  The displayed bars
centered at each estimate $\widehat{d}_T(u)$ represent 0.95 level confidence
intervals, based on the asymptotic distribution given by~(\ref{eq:CLTd}). Since
the asymptotic variance depends continuously on $d(u)$, we plug
$\widehat{d}_T(u)$ in to compute each interval length. Numerical computations
are done using the toolbox described in~\cite{fay-mouline-roueff-taqqu-2009}.
One can observe the difference between the two-side kernel estimator and the
recursive estimator. The former exhibits a uniform behavior along time with
border effects close to each boundaries of the interval $[0,1]$ (here we
dropped the values of $\widehat{d}_T(u)$ for $u<\band_T/2$ and $u >1-\band_T/2$
to avoid these border effects).  In contrast the latter exhibits a diminishing
then stabilizing variability along time. Thus it is better adapted for
estimating the right part of the interval.  It is interesting to note that the
choice of $L$ is crucial for this simulated example. This is due to the
presence of an autoregressive component leading to a strong positive
short-memory autocorrelation with a root close to the unit circle.  As a result
$d(u)$ is over estimated if a too low frequency band of scales is used (as in
the case $L=1$), which explains why the true value mostly lies out of the
corresponding confidence intervals.  On the other hand this bias diminishes
drastically as soon as $L\geq2$, but, for $L=3$, the confidence intervals are
larger since the normalizing term $\delta_{L,T}$ is larger.  This larger
variance is matched by the fact that the estimates are varying more widely
for $L=3$.
We made similar experiments for a tvARFIMA(0,$d$,0) process. In this case, this
bias is no longer observed for $L=1$.  We have also tried different values of
the bandwidth $\band_T$ which also influences the bias and the variability of
the estimates in the expected way. Finally we tested our procedures on longer
series to check the numerical tractability.  The computation of
$\widehat{\sigma}^2_{j,T}(u)$ from $X_{1,T},\dots,X_{T,T}$, with $T=2^{15}$
took less than 1 second for the kernel estimator and 7 seconds for the
recursive estimator with a 3.00GHz CPU. We note that the recursive version is
about ten times slower than the kernel estimator. On the other hand the
recursive estimator is adapted to \emph{online} computation, that is,
$\widehat{\sigma}^2_{j,T}(t)$ can be computed in a recursive fashion for each
new available observation $X_{t,T}$.

\subsection{Real data sets}

We now use real data sets made of a sample of realized log volatility of the
YEN versus USD exchange rate between June 1986 and September 2004. The realized
log volatility is represented in Figure~\ref{fig:raw_yen}.
\begin{figure}[htbp]
\centering
\includegraphics[width=0.6\textwidth]{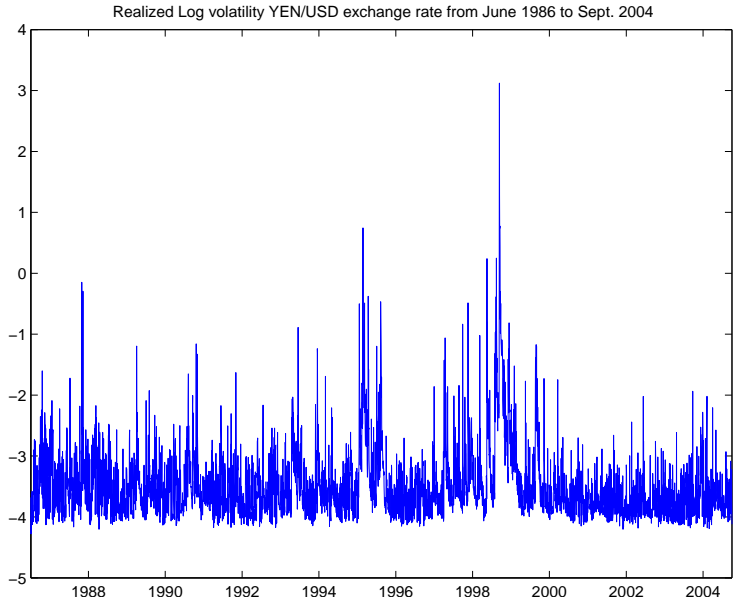} 
\caption{Realized log volatility of the YEN vs USD exchange rate from June 1986 to September 2004.}
\label{fig:raw_yen}
\end{figure}
The series length is $T=4470$, that is of the same order as the previously
simulated series ($T=2^{12}=4096$).  Viewing the simulated data as a benchmark,
we used approximately the same bandwidth parameter $\band_T=0.23$ and the same
sets of scales, namely $L=1,2,3$ with $\ell=3$ in the three cases. The
two-sided kernel estimators of the memory parameter are represented in the
upper part of Figure~\ref{fig:estim_yen_two_sided}.  As previously we also
display the corresponding local scalograms in the lower part of the same
figure.  We omit the results for the recursive estimator for brevity.  One can
observe that here as $L$ increases the estimates of $d(u)$ globally increases
which may indicate a negative bias at high frequencies. We only plot the
confidence intervals for the first 10 estimates for clarity. Indeed, in
contrast to the simulated case, they largely cover each other, which indicates
a less important bias. The green curve appears as a good compromise as in the
simulated example. It exhibits a 5 years periodic-like behavior, which seems to
indicate that the long memory parameter is not constant over time. This seems
to be in accordance with the findings of \cite{bauwens:wang:2008} who model
long-memory realized volatilities by a change of the model parameters from one
regime to another where the different regimes can be explained by the influence
of changing market factors (such as the Asian financial crisis of 1998).

\begin{figure}[htbp]
\centering
\includegraphics[width=0.8\textwidth]{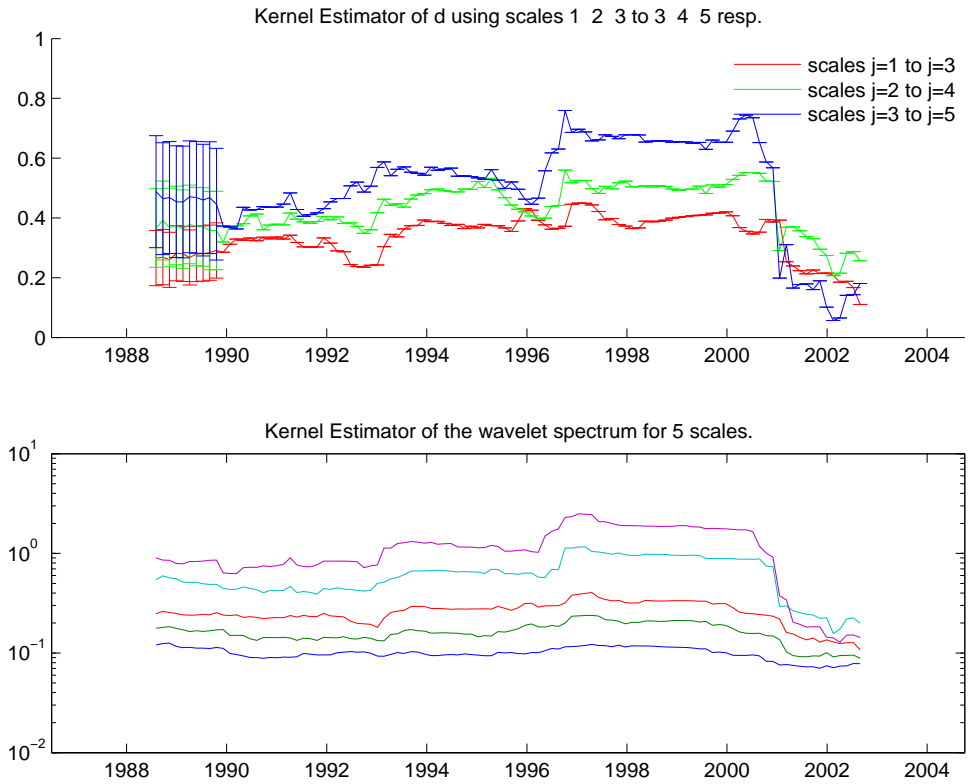}
\caption{Same as Figure~\ref{fig:estim_simu_two_sided} for the YEN vs USD exchange rate realized log volatility.}
\label{fig:estim_yen_two_sided}
\end{figure}

\section{Conclusion}\label{sec:6}

In this paper we have delivered a semi-parametric, hence fairly general,
approach for estimating the time-varying long-memory parameter $d(u)$ of a
locally stationary process (or stationary increment process). Apart from
modelling the singularity at zero frequency by the curve $d(u)$, we do not need
to model the time varying spectrum of the remaining part explicitly. Using a
wavelet log-regression estimator, already shown to be well-performing in the
stationary situation, continues to work well due to a localization of the
wavelet scalograms across time within each scale.

The development of our approach is based on a weakly stationary approximation at each
given time point $u$. As in the stationary case, due to the generality of our
semi-parametric spectral density not to be depending on only a finite number of
parameters (as in \cite{beran:2009}, e.g.), we need to concentrate our
attention to well estimating around frequency zero (where the amount of the
long-memory effect measured by $d$ is visible). So a slightly subtle choice of
considered scales for the log-regression has to be done: asymptotically we need
that our estimator involves more and more frequencies (i.e. scales) but with a
\emph{maximal frequency} tending to zero.  In the wavelet domain, this means
that the lowest scale used in the estimator will be chosen so that i) the
number of wavelet coefficients used in the estimator tends to infinity and ii)
this lowest scale itself tends (slowly) to infinity.
  


Simulations have shown that our estimator performs reasonably well beyond being attractive from the point of view of asymptotic theory. 
In our real data analysis example, we adopt the approach of \cite{bauwens:wang:2008} and of \cite{granger:ding:engle:1993}
to assume that realized volatilities of some exchange rates follow a long-memory model. We make the interesting observation
that for the observed series the long-memory parameter can clearly not be considered to be constant over time - which
suggests that in explaining the persistent correlation in this exchange data there are certainly periods of stronger
persistence followed by periods of weaker persistence.


\section*{Acknowledgements}

We acknowledge financial support from the TSIFIQ project of the ``Fondation Telecom'', from the  
IAP research network grant no P6/03 of the Belgian government (Belgian Science
Policy), and from the ``Projet d'Actions de Recherche Concert\'ees'' no 07/12-002 of the ``Communaut\'e
fran\c{c}aise de Belgique'', granted by the ``Acad\'emie universitaire Louvain''.

We also thank S.-H. Wang and L. Bauwens for providing the data example of Section 5.2 and for their helpful comments related
to this analysis. Finally we would like to thank two anonymous referees for their helpful comments in order to improve our paper.

\appendix

\section{Postponed proofs}
\begin{proof}[of Proposition~\ref{prop:local-scalogram}]
By \cite[Proposition~3]{moulines:roueff:taqqu:2007:jtsa}, there is a constant $C_1$ such that,
for all $j\geq0$ and all $\lambda\in[-\pi,\pi]$, 
\begin{equation}
\label{eq:fjBound}
|H_j( \lambda)|
 \leq  C_3\, 2^{j/2} \, |2^j\lambda|^M\,(1+2^j|\lambda|)^{-\alpha-M}\;.
\end{equation}
Applying~(\ref{eq:LocStatProcSpectralExt}),~(\ref{eq:tangentProcSpectral}),~(\ref{eq:wavcoeffLocStat})
and~(\ref{eq:wavcoeffStat}), we get, for any $u\in\R$, $j\geq0$ and $k\in\{0,\dots,T_j-1\}$,
\begin{equation}
  \label{eq:decompWaveletCoeff}
  W_{j,k;T}=W_{j,k}(u)+ R_{j,k}(u;T)\;,
\end{equation}
where
$$
R_{j,k}(u;T)=
\int_{-\pi}^\pi
\sum_{s\in\Z} \tilde{h}_{j,s}\left[A^0_{2^jk-s,T}(\lambda)-A(u;\lambda)\right]\rme^{\rmi\lambda(2^jk-s)}\;\rmd Z(\lambda)\;. 
$$
The main approximation result consists in bounding
$$
S_j(u;T)=\sum_{k=0}^{T_j-1}\gamma_{j,T}(k) R_{j,k}^2(u;T)
$$
and
$$
D_j(u;T)=\sum_{k=0}^{T_j-1}\gamma_{j,T}(k) W_{j,k}(u)R_{j,k}(u;T)\;.
$$
In the following $C$ denotes some multiplicative constant.
Using~(\ref{eq:AssumAtimeSmooth2}),~(\ref{eq:AssumAtimeSmooth}), and~(\ref{eq:AbsBoundWaveletfilter}) in Lemma~\ref{lem:filtre}, we have
$$
\left|\sum_{s\in\Z} \tilde{h}_{j,s}\left[A^0_{2^jk-s,T}(\lambda)-A(u;\lambda)\right]\rme^{\rmi\lambda(2^jk-s)}\right|\leq
C\;2^{jp}\;|\lambda|^{-D}\; \left\{2^{j/2}\left|2^jk/T-u\right|+2^{3j/2}/T\right\}\; .
$$
Recall that $D$ denotes an exponent less than $1/2$ which appears in the
Conditions~(\ref{eq:AssumAtimeSmooth2}) and~(\ref{eq:AssumAtimeSmooth}).  Using
$D<1/2$, we get
$$
\E\left[R_{j,k}^2(u;T)\right]\leq C\;2^{2jp}\;2^{3j}\;T^{-2}\;\left\{1+(k-Tu2^{-j})^2\right\} \; .
$$
Since we assumed $\alpha>1/2-d(u)$, we can take 
$D$ large enough so that $1-\alpha-d(u)<D<1/2$ (by adapting the constant $c$ appearing in the afore mentioned conditions).  
Hence we can assume in the following that
\begin{equation} 
  \label{eq:CondMandAlphaRemainder}
  M>d(u)-1/2 \quad \text{and}\quad d(u)+D+\alpha>1 \;.
\end{equation}
By~(\ref{eq:wavcoeffStatSpectral}) we also obtain that
$$
\left|\E\left[ W_{j,k}(u)R_{j,k}(u;T)\right]\right|\leq C\;2^{jp}\;\left\{2^{j/2}\left|2^jk/T-u\right|+2^{3j/2}/T\right\}\;
\int_{-\pi}^\pi\left|\tilde{H}_j(\lambda)A(u;\lambda)\right|\;|\lambda|^{-D}\;\rmd\lambda 
$$
Using~(\ref{eq:HjTilde}),~(\ref{eq:genLocSpecDens}),~(\ref{eq:genLocSpecDensSemiParam}), $f^*(u,\lambda)\leq C f^*(u,0)$
(by~(\ref{eq:genLocSpecDensSemiParamCond})), and~(\ref{eq:fjBound}), we further have
\begin{align*}
\int_{-\pi}^\pi\left|\tilde{H}_j(\lambda)A(u;\lambda)\right|\;|\lambda|^{-D}\;\rmd\lambda 
&\leq C\;\int_{-\pi}^\pi\left|H_j(\lambda)\right|\ \sqrt{f(u,\lambda)}\;|\lambda|^{-D}\;\rmd\lambda  \\
&\leq C\;\sqrt{f^*(u,0)}\;2^{j(d(u)+D-1/2)} \;,
\end{align*}
where we used that $\int_{\R} |\xi|^{M-d(u)-D}\ (1+|\xi|)^{-\alpha-M}\rmd(\xi)<\infty$
by~(\ref{eq:CondMandAlphaRemainder}). 
The last displays provide simple bounds for the expectations of $S_j$ and $D_j$.

To bound their variance, we use 
 \cite[Theorem~2, page~34]{rosenblatt:1985}, which yields
\begin{multline*}
\mathrm{Cov}\left(R_{j,k}^2(u;T),R_{j,k'}^2(u;T)\right)=
2\mathrm{Cov}^2\left(R_{j,k}(u;T),R_{j,k'}(u;T)\right)\\
+\mathrm{Cum}
\left(R_{j,k}(u;T),R_{j,k}(u;T),R_{j,k'}(u;T),R_{j,k'}(u;T)\right)
\end{multline*}
and
\begin{multline*}
\mathrm{Cov}\left(W_{j,k}(u)R_{j,k}(u;T),W_{j,k'}(u)R_{j,k'}(u;T)\right)\\=
\mathrm{Cov}\left(W_{j,k}(u),W_{j,k'}(u)\right)\mathrm{Cov}\left(R_{j,k}(u;T),R_{j,k'}(u;T)\right)\\
+\mathrm{Cov}\left(R_{j,k}(u;T),W_{j,k'}(u)\right)
\mathrm{Cov}\left(W_{j,k}(u),R_{j,k'}(u;T)\right) \\
+\mathrm{Cum}
\left(W_{j,k}(u),R_{j,k}(u;T),W_{j,k'}(u),R_{j,k'}(u;T)\right)
\;.
\end{multline*}

Let us first provide bounds of 
$\E\left[R_{j,k}(u;T)R_{j,k'}(u;T)\right]$ and $\E\left[ W_{j,k}(u)R_{j,k'}(u;T)\right]$
for $k,k'=0,\dots,T_j-1$.
Proceeding as previously, using~(\ref{eq:CondMandAlphaRemainder}),
we get (in fact the cases above $k=k'$ are particular cases)
$$
\left|\E\left[R_{j,k}(u;T)R_{j,k'}(u;T)\right]\right|\leq 
C\;2^{2jp}\;2^{3j}\;T^{-2}\;\left\{1+|k-Tu2^{-j}|\right\}\left\{1+|k'-Tu2^{-j}|\right\} \; .
$$
and
$$
\left|\E\left[ W_{j,k}(u)R_{j,k'}(u;T)\right]\right|\leq C\;2^{jp}\;\sqrt{f^*(u;0)}\;
2^{j(d(u)+D-1/2)}\;\left\{2^{j/2}\left|2^jk'/T-u\right|+2^{3j/2}/T\right\}\; \;.
$$
Using~(\ref{eq:cumBound}), we further get
\begin{multline*}
\left|\mathrm{Cum}
\left(R_{j,k}(u;T),R_{j,k}(u;T),R_{j,k'}(u;T),R_{j,k'}(u;T)\right)\right|\\
\leq
C\;2^{4jp}\;2^{6j}\;T^{-4}\;\left\{1+(k-Tu2^{-j})^2\right\}\left\{1+(k'-Tu2^{-j})^2\right\} \; ,
\end{multline*}
and, denoting by $B_j(u)$ the variance of the (weakly stationary) process
$\{W_{j,k}(u),\;k\in\Z\}$,  
\begin{multline*}
\left|\mathrm{Cum}
\left(W_{j,k}(u),R_{j,k}(u;T),W_{j,k'}(u),R_{j,k'}(u;T)\right)\right|\\
\leq C\;B_j(u)\,2^{2jp}\;2^{3j}\;T^{-2}\;\left\{1+|k-Tu2^{-j}|\right\}\;
\left\{1+|k'-Tu2^{-j}|\right\} 
\end{multline*}
Gathering these bounds, we obtain the same bound for
$\mathrm{Var}^{1/2}\left(S_j(u;T)\right)$ and $\E\left[S_j(u;T)\right]$ and
thus, using the definition of $\Gamma$ in~(\ref{eq:GammaDefQ}),
\begin{equation}
  \label{eq:Sbound}
\left|\E\left[S_j^2(u;T)\right]\right|^{1/2}\leq C\;2^{2jp}\;
2^{3j}\;T^{-2}\;\{\Gamma_0(u;j,T)+\Gamma_2(u;j,T)\} \;.  
\end{equation}
For $D_j(u;T)$, we obtain
$$
\left|\E\left[D_j(u;T)\right]\right|\leq C\;2^{jp}\;\sqrt{f^*(u;0)}\;2^{j(d(u)+D-1/2)}\;
2^{3j/2}\;T^{-1}\;\{\Gamma_0(u;j,T)+\Gamma_1(u;j,T)\} \;.
$$
We then obtain that
$\mathrm{Var}^{1/2}\left(D_j(u;T)\right)$ is at most
$$
C\;2^{jp}\;2^{3j/2}\;T^{-1}
\;\{\Gamma_0(u;j,T)+\Gamma_1(u;j,T)\}
\left\{B_j^{1/2}(u)+\sqrt{f^*(u;0)}\;2^{j(d(u)+D-1/2)}\right\} \;.
$$
Observe that by \cite[Theorem~1]{moulines:roueff:taqqu:2007:jtsa} we have, since
$M>d(u)-1/2$ and $\alpha>1/2-d(u)$, $B_j(u)\leq C\;f^*(u;0)\; 2^{2d(u)j}$. Hence, since 
$D<1/2$,   
\begin{equation}
  \label{eq:Dbound}
\left|\E\left[D_j^2(u;T)\right]\right|^{1/2}  \leq C\;2^{jp}\;\sqrt{f^*(u;0)}\;2^{j(3/2+d(u))}\;T^{-1}
\;\{\Gamma_0(u;j,T)+\Gamma_1(u;j,T)\} \;.
\end{equation}
By~(\ref{eq:EstimatorWeights}) and~(\ref{eq:decompWaveletCoeff}), we have
\begin{equation}
  \label{eq:StatApproxScalo}
  \widehat{\sigma}^2_{j,T}(u)=\widetilde{\sigma}^2_{j,T}(u)+S_j(u;T)+D_j(u;T) \; ,
\end{equation}
where $\widetilde{\sigma}^2_{j,T}(u)$ is defined in~(\ref{eq:tangentScalogram}).  
The bound~(\ref{eq:ScalApproxResult}) now follows from~(\ref{eq:Sbound}),~(\ref{eq:Dbound}),~(\ref{eq:StatApproxScalo}) and
Assumption~\ref{assump:weights}~(\ref{item:Locassump}). 
\end{proof}

\begin{proof}[of Theorem~\ref{theo:rate-param-scalo}]
By~(\ref{eq:defLocWavSpec}) and~(\ref{eq:normWeights}),
\begin{equation}
  \label{eq:SigmatildeMean}
\E\left[\widetilde{\sigma}^2_{j,T}(u)\right]=\E\left[W_{j,k}^2(u)\right]=\sigma^2_j(u)\;.
\end{equation}
Since the wavelet coefficients~(\ref{eq:wavcoeffStat}) are those of a
weakly stationary process, their behavior at large scales ($j\to\infty$) can be
studied using \cite[Theorem~1]{moulines:roueff:taqqu:2007:jtsa}.
By~\cite[Theorem~1]{moulines:roueff:taqqu:2007:jtsa}, since we assumed
(\ref{eq:genLocSpecDensSemiParamCond}) and $M>d(u)-1/2$ and
$\alpha>(1+\beta)/2-d(u)$, we obtain~(\ref{eq:BiasControl}).  In the following
we denote
$$
K^*_u=f^*(u,0)\constantePsiD(d(u))\;.
$$
We now provide a bound for
$$
\mathrm{Var}\left(\widetilde{\sigma}^2_{j,T}(u)\right)=
\sum_{k,k'=0}^{T_j-1}\gamma_{j,T}(k)\gamma_{j,T}(k')\mathrm{Cov}\left(W_{j,k}^2(u),W_{j,k'}^2(u)\right)=V_1+V_2\;,
$$
where the decomposition in $V_1+V_2$ follows from that of 
$\mathrm{Cov}\left(W_{j,k}^2(u),W_{j,k'}^2(u)\right)$ in
$$
2\mathrm{Cov}^2\left(W_{j,k}(u),W_{j,k'}(u)\right)+
\mathrm{Cum}\left(W_{j,k}(u),W_{j,k}(u),W_{j,k'}(u),W_{j,k'}(u)\right)\;.
$$
We easily obtain that
\begin{equation}
  \label{eq:V1exp}
V_1=2\int_{-\pi}^\pi\left|\Phi_{j,T}(\lambda;0,0)\right|^2\;\bD_j^{\star2}(\lambda)\;\rmd\lambda\;,  
\end{equation}
where $\bD_j$ denotes the spectral density of the weakly stationary process
$\{W_{j,k}\;,k\in\Z\}$, $\Phi_{j,T}$ is defined in~(\ref{eq:defweightsFourier})
and, for any $(2\pi)$-periodic function $g$, $g^{\star2}=g\star
g(\lambda)=\int_{-\pi}^\pi g(\lambda-\xi)g(\xi)\rmd \xi$. Moreover,
applying~(\ref{eq:cum-spectral}) with~(\ref{eq:wavcoeffStatSpectral})
and~(\ref{eq:defweightsFourier}), we get that $V_2$ can be expressed as
$$
\int_{[-\pi,\pi]^4}
\prod_{k=1}^4\left[\tilde{H}_j(\lambda_k)A(u;\lambda_k)\right]
\Phi_{j,T}(2^j(\lambda_1+\lambda_2);0,0)\Phi_{j,T}(2^j(\lambda_3+\lambda_4);0,0)
\hat{\kappa}_4(\lambda)\;\rmd\mu(\lambda)\;.
$$
Hence, bounding $\hat{\kappa}_4$, using~(\ref{eq:muDef}) and setting
$\lambda=\lambda_1+\lambda_2$, we have
\begin{equation}\label{eq:V2bound}
|V_2|\leq c_4\int_{-\pi}^\pi\left|\aleph_j^{\star2}(\lambda)\right|^2\;
\left|\Phi_{j,T}(2^j\lambda;0,0)\right|^2\;\rmd\lambda\;,
\end{equation}
where we set $\aleph_j(\lambda)=\tilde{H}_j(\lambda)\ A(u;\lambda)$. Observe
that $\|\aleph_j^{\star2}\|_\infty\leq\|\aleph_j\|_2^2=\|\bD_j\|_1$ and
$\|\bD_j^{\star2}\|_\infty\leq\|\bD_j\|_2^2$.  Now by
\cite[Theorem~1]{moulines:roueff:taqqu:2007:jtsa} we have, since $M\geq d(u)$
and $\alpha>1/2-d(u)$, $\|\bD_j\|_\infty= O\left(2^{2jd(u)}\right)$ (the
constants depend on $f^*(u;\cdot)$ only), which implies bounds of the same
order for $\|\bD_j\|_1$ and $\|\bD_j\|_2$.
Using~(\ref{eq:V1exp}),~(\ref{eq:V2bound}),
Assumption~\ref{assump:weights}(ii) with $i=i'=v=v'=0$ and observing that, by
$(2\pi)$ periodicity of $|\Phi_{j,T}(\lambda;0,0)|$, 
$$
\int_{-\pi}^\pi
\left|\Phi_{j,T}(2^j\lambda;0,0)\right|^2\;\rmd\lambda
=\int_{-\pi}^\pi\left|\Phi_{j,T}(\lambda;0,0)\right|^2\;\rmd\lambda\;,
$$
we finally get that
\begin{equation}
  \label{eq:VarTildeBound}
  \mathrm{Var}\left(\widetilde{\sigma}^2_{j,T}(u)\right)=O(2^{4jd(u)}\delta_{j,T}) \; .
\end{equation}
  Using~(\ref{eq:StatApproxScalo}) and~(\ref{eq:SigmatildeMean}),
  $\E\left[(\widehat{\sigma}^2_{j,T}(u)-K^*_u\;2^{2jd(u)})^2\right]$ is at most
\begin{multline*}
 C\;\left\{ 
\mathrm{Var}\left(\widetilde{\sigma}^2_{j,T}(u)\right)+\E\left[S_j^2(u;T)\right]+\E\left[D_j^2(u;T)\right] \right\}+O\left(2^{2
  (2d(u)-\beta) j}\right)\\
= O\left(2^{4jd(u)}\delta_{j,T}+2^{(6+4p)j}T^{-4} \delta_{j,T}^{-4} + 2^{(3+2p+2d(u))j}T^{-2} \delta_{j,T}^{-2}+2^{2 (2d(u)-\beta) j}\right)\;.
\end{multline*}
where we used~(\ref{eq:VarTildeBound}),~(\ref{eq:Sbound}),~(\ref{eq:Dbound}) and~(\ref{eq:NormalizedWeightsSums}).
Using~(\ref{eq:condConsistency}), the last display gives~(\ref{eq:ratesOp}).
\end{proof}

\begin{proof}[of Theorem~\ref{thm:clt-tang-scal}]
  Under the set of Assumptions (a), the proof
  immediately follows from~\cite[Theorem~2]{roueff-taqqu-2009b}. We now
  consider the  set of Assumptions (b). In this case, we rely on the Gaussian
  assumption.  The proof follows the lines
  of~\cite[Theorem~2]{moulines:roueff:taqqu:2007:fractals}, in which the
  stationary case is considered, \emph{i.e.} $\gamma_{j,T}(u)=1$.  We first
  observe that, for any $\mu=[\mu_0\,\,\dots\,\,\mu_{\ell}]^T\in\R^{\ell+1}$,
  we may write
$$
\mu^T\widetilde{S}_L(u)=\xi_L^T\Delta_L\xi_L\;,
$$
where $\xi_L$ is a Gaussian vector with entries 
$\left(W_{L+i,k}(u)\right)_{0\leq i\leq\ell,\,0\leq k\leq T_{L+i}}$ 
and $\Delta_L$ is the diagonal matrix with diagonal entries 
$\left(2^{-2Ld(u)}\delta_{L,T}^{-1/2}\mu_i\gamma_{L+i,T}(u)\right)_{0\leq i\leq\ell,\,0\leq k\leq T_{L+i}}$.
We may thus apply \cite[Lemma~12]{moulines-roueff-taqqu-2008}.

To obtain~(\ref{eq:S_jCLT}), it is thus sufficient to show that 
\begin{equation}\label{eq:S_jCLTcond1}
\rho(\Delta_L)\rho\left(\mathrm{Cov}(\xi_L)\right)\to0\;,
\end{equation}
where $\rho(A)$ denotes the spectral radius of $A$, and
\begin{equation}\label{eq:S_jCLTcond2}
\mathrm{Cov}(\mu^T\widetilde{S}_L(u))\to (f^*(u,0))^2\mu^T\Sigma\mu\;.
\end{equation}

We have, by~(\ref{eq:deltaDef}) and Assumption~\ref{assump:weights}(i),
$$
\rho(\Delta_L)\leq 2^{-2Ld(u)}\delta_{L,T}^{-1/2}
\max_{0\leq i\leq\ell}|\mu_i|\max_{0\leq i\leq\ell}\delta_{L+i,T}=o\left(2^{-2Ld(u)}\right)\;.
$$
Using~\cite[Lemma~6]{moulines:roueff:taqqu:2007:fractals},  
\cite[Lemma~11]{moulines-roueff-taqqu-2008} and that $\bD_{L+i}$ is the spectral density of the process
$\{W_{L+i,k}(u),\,k\in\Z\}$, we have 
$$
\rho\left(\mathrm{Cov}(\xi_L)\right)\leq
\sum_{i=0}^\ell\rho\left(\mathrm{Cov}([W_{L+i,k}(u),\,k=0,\dots,T_{L+i}])\right)
\leq2\pi\sum_{i=0}^\ell\|\bD_{L+i}\|_\infty\;.
$$
By  \cite[Theorem~1]{moulines:roueff:taqqu:2007:jtsa}, since we assumed $M\geq d(u)$ and $\alpha>1/2-d(u)$, we have
$\|\bD_{L+i}\|_\infty=O(2^{2Ld(u)})$. This with the last two displays implies~(\ref{eq:S_jCLTcond1}). 

We now compute the asymptotic covariance matrix of $\widetilde{S}_L(u)$. 
Let $0\leq j'\leq j$. Using~(\ref{eq:calT}) and the Gaussian assumption, we have
\begin{align*}
\mathrm{Cov}\left(\widetilde{\sigma}^2_{j,T}(u),\widetilde{\sigma}^2_{j',T}(u)\right)&=
\sum_{k=0}^{T_j-1}\sum_{k'=0}^{T_{j'}-1}
\gamma_{j,T}(k)\gamma_{j',T}(k')\mathrm{Cov}\left(W_{j,k}^2(u),W_{j',k'}^2(u)\right)\\
&\hspace{-2.5cm}=2\sum_{v=0}^{2^{j-j'}-1}\sum_{k=0}^{T_j-1}\sum_{l\in\calT_j(j-j',v)}
\gamma_{j,T}(k)\gamma_{j',T}(l2^{j-j'}+v)\mathrm{Cov}^2\left(W_{j,k}(u),W_{j',l2^{j-j'}+v}(u)\right)\;.
\end{align*}
Using \cite[Corollary~1]{moulines:roueff:taqqu:2007:jtsa}, we have 
$$
\mathrm{Cov}\left(W_{j,k}(u),W_{j',l2^{j-j'}+v}(u)\right)=
\int_{-\pi}^\pi\bD_{j,j-j',v}(\lambda)\rme^{\rmi\lambda(k-l)}\,\rmd\lambda \;,
$$
where $\bD_{j,j-j'}=[\bD_{j,j-j',v}]_{v=0,\dots,2^{j-j'}-1}$ denotes the $2^{j-j'}$-dimensional cross-spectral density
between $W_{j,k}(u)$ and $[W_{j',l2^{j-j'}+v}(u)]_{v=0,\dots,2^{j-j'}-1}$. It follows from the last two displays and~(\ref{eq:defweightsFourier}) that 
\begin{align*}
\mathrm{Cov}\left(\widetilde{\sigma}^2_{j,T}(u),\widetilde{\sigma}^2_{j',T}(u)\right)
&=2\sum_{v=0}^{2^{j-j'}-1}\int_{-\pi}^\pi\Phi_{j,T}(\lambda;0,0)\overline{\Phi}_{j,T}(\lambda;j-j',v)
\;\widetilde{\bD}_{j,j-j',v}(\lambda)\;\rmd\lambda\;,
\end{align*}
where
$$
\widetilde{\bD}_{j,j-j',v}(\lambda)=\int_{-\pi}^\pi \bD_{j,j-j',v}(\xi)\overline{\bD}_{j,j-j',v}(\xi-\lambda)\rmd\xi\;.
$$
By \cite[Theorem~1(b)]{moulines:roueff:taqqu:2007:jtsa}, since we assumed $M\geq d(u)$ and $\alpha>1/2-d(u)$,
using~(\ref{eq:genLocSpecDensSemiParamCond}), we have, for $j=L+i$ and $j'=L+i'$ with $i'\leq i$ fixed,
$$
\|2^{-2d(u) j}\bD_{j,j-j'}-f^*(u,0)\bD_{\infty,i-i'}(\cdot;d(u))\|_\infty\to 0 \; .
$$
The last three displays,~(\ref{eq:Sigmadef}), Lemma~\ref{lem:CovLim} and Assumption~\ref{assump:weights} yield
$$
\mathrm{Cov}\left(\widetilde{S}_{L,T}(u)\right)\to (f^*(u,0))^2\Sigma \;,
$$
and hence~(\ref{eq:S_jCLTcond2}).
\end{proof}

\begin{proof}[of Theorem~\ref{theo:clt-param-est}]
We first show that
\begin{equation}\label{eq:hatS_jAsymNorm}
\delta_{L,T}^{-1/2}
\left(\;2^{-2Ld(u)}
\left[  \begin{array}[]{c}
\widehat{\sigma}^2_{L,T}(u)\\
\widehat{\sigma}^2_{L+1,T}(u)\\
\vdots\\
\widehat{\sigma}^2_{L+\ell,T}(u)
  \end{array}
\right]- K^*_u \left[
\begin{array}[]{c}
1\\
2^{2d(u)}\\
\vdots\\
2^{2\ell d(u)}
\end{array}
\right]\right)\Rightarrow \calN\left(0,(f^*(u,0))^2\Sigma(u)\right) \;.
\end{equation}
Observe that the weak convergence~(\ref{eq:hatS_jAsymNorm}) is the same as~(\ref{eq:hatS_jCLT}) except for the centering term.
Relation~(\ref{eq:hatS_jCLT}) is valid since the assumptions of Corollary~\ref{cor:clt-local-scalogram} hold.
Applying $\delta_{L,T}\to0$, Proposition~\ref{prop:local-scalogram} and the left-hand side condition of~(\ref{eq:approxScalandBiasNeglict}), we have that,
for any $j=L+i$ with a fixed $i=0,\dots,\ell$, 
$$
\delta_{L,T}^{-1/2}2^{-2Ld(u)}
\E\left[\widehat{\sigma}^2_{j,T}(u)\right]=\delta_{L,T}^{-1/2}2^{-2Ld(u)}\E\left[\widetilde{\sigma}^2_{j,T}(u)\right]+o(1)\;.
$$
The bias control~(\ref{eq:BiasControl}) and the right-hand side condition of~(\ref{eq:approxScalandBiasNeglict}) then imply
$$
\delta_{L,T}^{-1/2}2^{-2Ld(u)}
\E\left[\widehat{\sigma}^2_{j,T}(u)\right]=\delta_{L,T}^{-1/2}f^*(u,0)\constantePsiD(d(u))2^{2id(u)}+o(1)\;.
$$
This, with~(\ref{eq:hatS_jCLT}) gives the weak convergence~(\ref{eq:hatS_jAsymNorm})\;.

The convergence~(\ref{eq:CLTd}) now follows from~(\ref{eq:hatS_jAsymNorm}) by applying the $\delta$-method as in
\cite[Proposition~3]{moulines:roueff:taqqu:2007:fractals}. Indeed, define 
$$
g(x)=\sum_{i=0}^\ell w_i\log(x_i)\quad\text{for all}\quad x=[x_0\,\,\dots\,\,x_\ell]^T\;.
$$
Observe that, by~(\ref{eq:propertyw}) and~(\ref{eq:definition:estimator:regression}), we have 
$$
g\left(2^{-2Ld(u)}[\widehat{\sigma}^2_{L,T}(u)\,\,\widehat{\sigma}^2_{L+1,T}(u)\,\,\dots\,\,\widehat{\sigma}^2_{L+\ell,T}(u)]^T\right)=
\widehat{d}_T(L)
$$
and
$$
g\left(f^*(u,0)\constantePsiD(d(u))[1\,\,2^{2d(u)}\,\,\dots\,\,2^{2\ell d(u)}]^T\right)= d(u)\;.
$$
Thus~(\ref{eq:CLTd}) follows from~(\ref{eq:hatS_jAsymNorm}) by computing the gradient of $g$ at the centering term,
$$
\nabla g\left(f^*(u,0)\constantePsiD(d(u))[1\,\,2^{2d(u)}\,\,\dots\,\,2^{2\ell d(u)}]^T\right)= 
\frac{[w_0\,\,w_12^{-2d(u)}\,\,\dots\,\,w_\ell2^{-2\ell d(u)}]^T}{f^*(u,0)\constantePsiD(d(u))}\;.
$$
\end{proof}

\section{Technical lemmas}\label{sec:technical-lemmas}
\begin{lem}\label{lem:filtre}
Assume~\ref{item:Wreg}--\ref{item:MIM}. 
Let $h_{j,\cdot}$ the wavelet detail filter at scale index $j$ and $\tilde{h}_{j,\cdot}$ any factorization of it by
$\diffop^p$ with $p\in\{0,\dots, M\}$. Then we have 
  \begin{equation}
    \label{eq:AbsBoundWaveletfilter}
\sum_{s\in\Z}|\tilde{h}_{j,s}| \leq C \; 2^{j(p+1/2)}\quad\text{and}\quad
\sum_{s\in\Z}(1+|s|) \; |\tilde{h}_{j,s}| \leq C \; 2^{j(p+3/2)}\;.    
  \end{equation}
\end{lem}
\begin{lem}\label{lem:CovLim}
Suppose Assumption~\ref{assump:weights} holds. 
Let $i,i'\geq0$, $v\in\{0,\dots,2^{i}-1\}$ and $v'\in\{0,\dots,2^{i'}-1\}$.
Define, for any $(2\pi)$-periodic function $g$,
$$
I_T(g)=\delta_{j,T}^{-1}
\int_{-\pi}^\pi\Phi_{j,T}(\lambda;i,v)\overline{\Phi_{j,T}(\lambda;i',v')}\;g(\lambda)\,\rmd\lambda\;.
$$
Then the two following assertions hold.
\begin{enumerate}[(i)]
\item  \label{eq:unifContAsympVar}
 If $h\to g$ in $L^\infty([-\pi,\pi])$, then
$\displaystyle \sup_{T\geq0}\left|I_T(h)-I_T(g)\right|\to 0$.
\item    \label{eq:AsympVarLim}
  If $g\in L^\infty([-\pi,\pi])$ is continuous at zero, then, as $T\to\infty$,
$\displaystyle I_T(g)\to V(i,v;i',v')\, g(0)$.
\end{enumerate}
\end{lem}
\begin{proof}
By linearity of $I_T$, we may take $g=0$ to prove Assertion~(\ref{eq:unifContAsympVar}).
We have, by the Cauchy-Schwarz inequality
$$
\left|I_T(h)\right|\leq\|h\|_\infty\left[\delta_{j,T}^{-1/2}\|\Phi_{j,T}(\cdot;i,v)\|_2\right]
\;\left[\delta_{j,T}^{-1/2}\|\Phi_{j,T}(\cdot;i',v')\|_2\right]\;.
$$ 
Using Assumption~\ref{assump:weights}(ii), the terms between brackets are
bounded independently of $j$ and we obtain~(\ref{eq:unifContAsympVar}).

We now prove~(\ref{eq:AsympVarLim}).
By linearity of $I_T$, we may assume $g(0)=1$. By Assumption~\ref{assump:weights}(ii), we have $I_T(1)\to V(i,v;i',v')$.  On
the other hand, we have, for any $\eta>0$
\begin{align*}
\left|I_T(g)-I_T(1)\right|&=\left|I_T((g-1)\1_{[-\eta,\eta]}+(g-1)\1_{[-\eta,\eta]^c})\right|\\
&\leq \left|I_T((g-1)\1_{[-\eta,\eta]})\right|+\left|I_T((g-1)\1_{[-\eta,\eta]^c})\right|\;.
\end{align*}
Observe that by continuity of $g$ at the origin, $\|(g-1)\1_{[-\eta,\eta]}\|_\infty\to0$ as $\eta\to\infty$. 
By~(\ref{eq:unifContAsympVar}), we get $\left|I_T((g-1)\1_{[-\eta,\eta]})\right|\to0$ as $\eta\to\infty$.
It thus only remains to show that $\left|I_T((g-1)\1_{[-\eta,\eta]^c})\right|\to0$ for any
$\eta>0$. This follows from the bound
$$
\left|I_T((g-1)\1_{[-\eta,\eta]^c})\right|\leq \|g-1\|_1 
\left[\delta_{j,T}^{-1/2} \sup_{\eta\leq|\lambda|\leq\pi}\left|\Phi_{j,T}(\lambda;i,v)\right|\right]
\left[\delta_{j,T}^{-1/2} \sup_{\eta\leq|\lambda|\leq\pi}\left|\Phi_{j,T}(\lambda;i',v')\right|\right]\;,
$$
and by applying Assumption~\ref{assump:weights}(iii).
\end{proof}

\begin{lem}\label{lem:lipshitz_with_log}
For any $a>0$ and $b>0$, there exists $c>0$ such that
$$
\left|z^\alpha-1\right|\leq c \,\{1+\log(|z|)\}\,\alpha\quad\text{for all}\quad\alpha\in[0,a],\,z\in\C\quad\text{with}\quad|z|\leq b\;.
$$
\end{lem}

\begin{lem}\label{lem:two-sided-kernel}
Assume one of the following.
\begin{enumerate}[(K-1)]
\item\label{item:ind} $K=\1_{[-1/2,1/2]}$.
\item\label{item:comp} $K$ is compactly supported and $|\hat{K}(\xi)|=o(|\xi|^{-3/2})$ as $|\xi|\to\infty$, 
where $\hat{K}$ denotes the Fourier transform of $K$. 
\item\label{item:exp} $K$ is integrable, $\hat{K}$ has an exponential decay, \emph{i.e.} for some
  $c>0$, $|\hat{K}(\xi)|=O\left(\exp(-c|\xi|)\right)$ as 
  $|\xi|\to\infty$, $K(t)=O(|t|^{-p_0})$  as $|t|\to\infty$ for some $p_0>3$, the derivative $K'$ of $K$ satisfies
  $|K'(t)|=O(|t|^{-p_1})$ as $|t|\to\infty$ for some $p_1>1$ 
  and $T_j\exp(-c' \band_{T}T_j)=O(1)$ for any $c'>0$.  
\end{enumerate}
Suppose that $\band_{T}\to0$ and that $j$ depends on $T$ so that $T_j\band_{T}\to\infty$ as $T\to\infty$. 
Then, for weights given by~(\ref{eq:KernelWeights}), Assumption~\ref{assump:weights} is satisfied with 
\begin{align}
  \label{eq:kerneldelta}
&\delta_{j,T}\sim \frac{\|K\|_\infty}{\|K\|_1} (\band_{T}T_j)^{-1} \\
  \label{eq:kernelV}
& V(i,v;i',v')= 2\pi\,\frac{\|K\|_2^2}{\|K\|_1\|K\|_\infty} \; 2^{-i-i'}\,,\quad i,i'\geq0,\;0\leq v<2^{i},\,0\leq v'<2^{i'}\;. 
\end{align}
\end{lem}
\begin{proof}
For convenience, we will omit the subscripts $_T$ and $_{j,T}$ in this proof section when no ambiguity arises.
Under (K-\ref{item:ind}), one has $\rho=\band T_j+O(1)$.
Under (K-\ref{item:comp}), $K$ is uniformly continuous on its compact support $S$ and, since $u\in(0,1)$, $\band \to0$ and
$T_j\band\to\infty$, $S$ eventually falls between the extremal points of $\{(uT_j-k) / (\band T_j),\,k=0,\dots,T_j-1\}$. Thus,
$$
(\band T_j)^{-1}\sum_{k=0}^{T_j-1}K( (uT_j-k) / (\band T_j) )\to \int_S K(s)\;\rmd s=\|K\|_1\;.
$$
Under (K-\ref{item:exp}), using that $|K'(t)|\leq c(1+|t|)^{-p_1}$ for some $p_1>1$ and $c>0$, we get
\begin{align*}
(\band T_j)^{-1}\sum_{k=0}^{T_j-1}K( (uT_j-k) / (\band T_j) )-\int_{(u-1)/\band}^{u/\band}K(s)\;\rmd s
&=O\left((\band T_j)^{-2}\sum_{l=0}^{T_j}(1+l/(\band T_j))^{-p_1}\right)\\
&=O\left((\band T_j)^{-1}\right)\;.
\end{align*}
Hence the last three displays yield that,  in all cases,
\begin{equation}
  \label{eq:kernelrho}
\rho_{j,T}\sim \|K\|_1 (\band_{T}T_j) \;.  
\end{equation}
The asymptotic equivalence~(\ref{eq:kerneldelta}) then follows from the definitions~(\ref{eq:KernelWeights})
and~(\ref{eq:deltaDef}), and we obtain Assumption~\ref{assump:weights}(\ref{item:delatAssump}) by~(\ref{eq:Tjasymp}).

Let us now prove that Assumption~\ref{assump:weights}(\ref{item:Vassump}) holds under (K-\ref{item:ind}), (K-\ref{item:comp})
and (K-\ref{item:exp}), successively. Note that, by definition of~(\ref{eq:defweightsFourier}), we have
\begin{equation}
  \label{eq:ConvDisc}
\int_{-\pi}^\pi \Phi_{j,T}(\lambda;i,v)\overline{\Phi_{j,T}(\lambda;i',v')}\,\rmd\lambda
=2\pi\sum_{l\in \calT_j(i,v)\cap\calT_j(i',v')}\gamma_{j-i,T}(2^{i} l+v)\gamma_{j-i',T}(2^{i'} l+v')\;.
\end{equation}
Under (K-\ref{item:ind}), using $2^{-i}T_{j-i}\sim2^{-i'}T_{j-i'}\sim T_j$ by~(\ref{eq:deltaDef}), 
$\band T_j\to\infty$ and $\band\to0$, we easily get that
the supports of the sequences $\{\gamma_{j-i,T}(2^{i} l+v),\,l\geq0\}$ and $\{\gamma_{j-i',T}(2^{i'} l+v),\,l\geq0\}$
are eventually included in $\calT_j(i,v)\cap\calT_j(i',v')$ and their intersection is of length asymptotically equivalent to $\band T_j$.
Hence, using~(\ref{eq:ConvDisc}),~(\ref{eq:kerneldelta}) and~(\ref{eq:kernelrho}) with $\|K\|_1=\|K\|_\infty=1$, we obtain
that, in this case,  
$$
\delta^{-1}\int_{-\pi}^\pi \Phi_{j,T}(\lambda;i,v)\overline{\Phi_{j,T}(\lambda;i',v')}\,\rmd\lambda\sim 
2\pi\,\frac{T_j^2}{T_{j-i}T_{j-i'}}\;.
$$
By~(\ref{eq:Tjasymp}), this is Assumption~\ref{assump:weights}(\ref{item:Vassump}) with
$V(i,v;i',v')=2\pi2^{-i-i'}$ which coincides with~(\ref{eq:kernelV}) under (K-\ref{item:ind}).

Under (K-\ref{item:comp}), we proceed by interpreting the sum in~(\ref{eq:ConvDisc}) as
a Riemann approximation of $\int K^2$ up to a normalization factor.
For $l\in\calT_j(i,v)\cap\calT_j(i',v')$, we approximate
\begin{align*}
J_l&=(\band T_j)^{-1}\rho_{j-i,T}\rho_{j-i',T}\gamma_{j-i,T}(2^{i} l+v)\gamma_{j-i',T}(2^{i'} l+v')\\
&=(\band T_j)^{-1}\,K(\{uT_{j-i}-(2^il+v)\}/\{\band T_{j-i}\})\,K(\{uT_{j-i'}-(2^{i'}l+v')\}/\{\band T_{j-i'}\})\;,
\end{align*}
by the local average
$$
\tilde{J}_l=\int_{I_l}K^2(s)\,\rmd s \;,
$$
where $I_l$ is defined as the interval $[\{uT_j-(l+1)\}/\{\band T_j\},\{uT_j-l\}/\{\band T_j\}]$.
Observe that
$$
\sup_{s\in I_l}\left|s-\{uT_{j-i}-(2^il+v)\}/\{\band T_{j-i}\}\right|\leq 
\frac1{\band T_j}
+\left|\frac1{\band T_j}-\frac{2^i}{\band T_{j-i}}\right|\,|l- uT_j|
+u\frac{|T_{j-i}-2^iT_j|}{\band T_{j-i}}
\frac{|v|}{\band T_{j-i}}\;.
$$
Using~(\ref{eq:Tjasymp}), $i,v=O(1)$ and $l=O(T_j)$, we obtain, for any fixed integers $i$ and $v$,
\begin{equation}
  \label{eq:distanceRiemannApprox}
\sup_{0\leq l \leq 2 T_j}\sup_{s\in I_l}\left|s-\{uT_{j-i}-(2^il+v)\}/\{\band T_{j-i}\}\right|= O((\band T_j)^{-1})\;,
\end{equation}
and the same holds if $i,v$ is replaced by $i',v'$. Note that 
$$
\calT_j(i,v)\cap\calT_j(i',v')=\{0,1\dots,\{2^{-i}(T_{j-i}-v)\}\wedge\{2^{-i'}(T_{j-i'}-v)\}-1\}\;,
$$
which, by~(\ref{eq:Tjasymp}) and the fact that $K$ is compactly supported, is eventually 
contained in $\{0,1,\dots,2 T_j\}$ and eventually contains 
the set of $l$'s such that $\tilde{J}_l\neq0$, which is of size $O(\band T_j)$. 
By~(\ref{eq:distanceRiemannApprox}), we also see that, out of a set of length  $O(\band T_j)$, both $J_l$ and $\tilde{J}_l$
vanish. Hence we have
$$
\sum_{l\in\calT_j(i,v)\cap\calT_j(i',v')}\left|J_l-\tilde{J}_l\right|=O\left(\band T_j \sup_l|J_l-\tilde{J}_l|\right)\;. 
$$
Using~(\ref{eq:distanceRiemannApprox}) and the uniform continuity of $K$, there exists a constant $c$ such that 
$$
\sup_l|J_l-\tilde{J}_l|\leq (\band T_j)^{-1}\sup_{|s-t|,|s-t'|\leq c/(\band T_j)}\left|K^2(s)-K(t)K(t')\right|=o((\band T_j)^{-1})\;.
$$
The last two displays,~(\ref{eq:ConvDisc}) and the definitions of $J_l$ and $\tilde{J}_l$ thus yield
\begin{equation}
  \label{eq:PSunderKind}
 \frac{\rho_{j-i,T}\rho_{j-i',T}}{\band T_j}\int_{-\pi}^\pi \Phi_{j,T}(\lambda;i,v)\overline{\Phi_{j,T}(\lambda;i',v')}\,\rmd\lambda
\sim 2\pi\int K^2(s)\,\rmd s=2\pi\|K\|_2^2\;.
\end{equation}
By~(\ref{eq:kerneldelta}) and~(\ref{eq:kernelrho}), this gives Assumption~\ref{assump:weights}(\ref{item:Vassump}) with
$V(i,v;i',v')$ given by~(\ref{eq:kernelV}).

Under (K-\ref{item:exp}), we proceed similarly but we can no longer use that $K$ has a compact support. 
Instead we use that $K$ is bounded and $|K'(t)|\leq c'(3+|t|)^{-p_1}$ for some $p_1>1$ and $c'>0$ and thus, 
for any $c>0$, as soon as $(c+1)/(\band T_j)\leq 1$,
$$
\sup_{s\in I_l}\sup_{|t-s|,|t'-s|\leq c/(\band T_j)}\left|K^2(s)-K(t)K(t')\right|\leq c''\,(\band T_j)^{-1}
(2+|uT_j-l|/(\band T_j))^{-p}\;.
$$
With~(\ref{eq:distanceRiemannApprox}) and since the length of
$\calT_j(i,v)\cap\calT_j(i',v')$ is $O(T_j)$, we get
$$
\sum_{l\in\calT_j(i,v)\cap\calT_j(i',v')}\left|J_l-\tilde{J}_l\right|=
O\left((\band T_j)^{-2}\sum_{k=0}^{T_j}(1+k/(\band T_j))^{-p}\right)
=O\left((\band T_j)^{-1}\right)\;. 
$$
Moreover
$$
\sum_{l\in\calT_j(i,v)\cap\calT_j(i',v')}\tilde{J}_l=\int_{-u'/\band}^{u/\band}K^2(s)\,\rmd s\to\|K\|_2^2\;,
$$
where $u'=[\{2^{-i}(T_{j-i}-v)\}\wedge\{2^{-i'}(T_{j-i'}-v')\}]/T_j-u\to1-u$ by~(\ref{eq:Tjasymp}).
This yields~(\ref{eq:PSunderKind}) as in the previous case
and thus the same conclusion holds.

Let us now show that Assumption~\ref{assump:weights}~(\ref{item:neglicAssump}) holds under (K-\ref{item:ind}), (K-\ref{item:comp})
and (K-\ref{item:exp}), successively.
Under (K-\ref{item:ind}), we have
$$
\left|\Phi(\lambda;i,v)\right|=\rho^{-1}_{j-i,T}\left|\sum_{k=1}^{N}\rme^{\rmi k\lambda}\right|\;,
$$
where $N=N_{j,T}$ denotes the number of $l\in\calT_j(i,v)$ such that $\gamma_{j-i,T}(2^il+v)>0$. 
Since the Dirichlet kernel satisfies
$$
|D_N(\lambda)|=\left|\sum_{k=1}^{N}\rme^{\rmi k\lambda}\right|=\left|\frac{\sin(\lambda N/2)}{\sin(\lambda /2)}\right|\;,
$$
we observe that, for any $\eta>0$,
$\sup_{N\geq1}\sup_{\lambda\in[\eta,2\pi-\eta]}\left|D_N(\lambda)\right|<\infty$.
Hence, with~(\ref{eq:kerneldelta}) and~(\ref{eq:kernelrho}), we obtain Assumption~\ref{assump:weights}~(\ref{item:neglicAssump}).

Under (K-\ref{item:comp}) and (K-\ref{item:exp}), using that $K(t)=(2\pi)^{-1}\int\hat{K}(\xi)\,\rme^{\rmi \xi t}\rmd \xi$, we get
$$
  \Phi(\lambda;i,v)=(2\pi\rho_{j-i,T})^{-1}\int_{-\infty}^\infty\;\hat{K}(\xi)\;\rme^{\rmi\xi(uT_{j_i}-v)/(\band T_{j-i})}\; 
\sum_{l\in\tilde{\calT}_j(i,v)} \rme^{\rmi l(\lambda+2^i\xi/(\band T_{j-i}))}\, \rmd \xi \;,
$$
where $\tilde{\calT}_j(i,v)$ denotes the set of all $l\in\calT_j(i,v)$ such that $\gamma_{j-i,T}(2^il+v)$ does not vanish. 
Denote the length of $\tilde{\calT}_j(i,v)$ by $N=N_{j,T}$ as in the previous case. We thus obtain
$$
\left|\Phi(\lambda;i,v)\right|\leq (2\pi\rho_{j-i,T})^{-1}\int_{-\infty}^\infty\;\left|\hat{K}(\xi)\right|
\;|D_N(\lambda+2^i\xi/(\band T_{j-i}))|\;\rmd \xi \;.
$$
Let $\eta>0$.
Splitting the above integral as $\int_{-\infty}^\infty=\int_{2^i|\xi|/(\band T_{j-i})\leq \eta/2}+\int_{2^i|\xi|/(\band
  T_{j-i})>\eta/2}$, we obtain
\begin{align*}
\sup_{\lambda\in[\eta,\pi]}\left|\Phi(\lambda;i,v)\right|\leq 
&(2\pi\rho_{j-i,T})^{-1}\sup_{|\lambda|\in[\eta/2,\pi+\eta/2]}|D_N(\lambda)|\\
&+(2\pi\rho_{j-i,T})^{-1}\|D_N\|_\infty\int_{2^i|\xi|/(\band T_{j-i})>\eta/2} \left|\hat{K}(\xi)\right|\rmd\xi\;.
\end{align*}
Now, we have, for $\eta$ small enough, $\sup_{N\geq1}\sup_{|\lambda|\in[\eta/2,\pi+\eta/2]}|D_N|(\lambda)<\infty$,
$\|D_N\|_\infty\leq N$ and, under (K-\ref{item:comp}), $N=O(\band T_j)$ and
$\int_{2^i|\xi|/(\band T_{j-i})>\eta/2} \left|\hat{K}(\xi)\right|\rmd\xi=o((\band T_{j})^{-1/2})$, which, with the previous
display,~(\ref{eq:kerneldelta}) and~(\ref{eq:kernelrho}), implies Assumption~\ref{assump:weights}~(\ref{item:neglicAssump}).
Under (K-\ref{item:exp}), the same conclusion holds using that $N=O(T_j)$, 
$\int_{2^i|\xi|/(\band T_{j-i})>\eta/2} \left|\hat{K}(\xi)\right|\rmd\xi=O(\exp(-c2^{-i-1}\eta\band T_{j}))$ 
and $T_j\exp(-c'\band T_{j-i})=O(1)$ with $c'=c2^{-i-1}\eta$\;.

Finally we show that Assumption~\ref{assump:weights}~(\ref{item:Locassump}) holds under (K-\ref{item:ind}), (K-\ref{item:comp})
and (K-\ref{item:exp}), successively.
Using the definition~(\ref{eq:GammaDefQ}) and~(\ref{eq:Tjasymp}), we get, for some positive constant $C$,
$$
\Gamma_q(u;j,T)\leq C\frac{(\band T_j)^q}{\rho_{j,T}}\sum_{k=0}^{T_j-1}K_q((uT_j-k)/(\band T_j)) + O(\Gamma_0(u;j,T))\;,
$$
where $K_q(x)=K(x)|x|^q$. By definition of $\rho_{j,T}$, one has $\Gamma_0(u;j,T)=O(1)$. Under (K-\ref{item:ind}) and
(K-\ref{item:comp}), $K_q$ is bounded and compactly supported, so that $\sum_{k}K_q((uT_j-k)/(\band T_j))=O(\band T_j)$. This,
with~(\ref{eq:kernelrho}) and the previous display, implies~(\ref{eq:NormalizedWeightsSums}) for all $q\ge0$.
Hence, to conclude the proof, it only remains to show that, for $q=1,2$, under (K-\ref{item:exp}),
$$
\sum_{k=0}^{T_j-1}K_q((uT_j-k)/(\band T_j)) =O(\band T_j)\;.
$$ 
Using that $K(x)=O(|x|^{-p_0})$  as $x\to\pm\infty$, and $q\leq2$, we separate the sum $\sum_{k=0}^{T_j-1}$ in
$\sum_{|uT_j-k|\leq \band T_j}$ for which $K_q((uT_j-k)/(\band T_j))$ is $O(1)$ and $\sum_{|uT_j-k|> \band T_j}$  for which
$K_q((uT_j-k)/(\band T_j))$ is $O(|(uT_j-k)/(\band T_j)|^{2-p_0})$. Hence, we get
$$
\sum_{k=0}^{T_j-1}K_q((uT_j-k)/(\band T_j)) =O\left(\band T_j\right)+O\left(\band T_j^{p_0-2} \sum_{l\geq \band T_j-1}l^{2-p_0}\right)\;.
$$
Observing that $\band T_j\to\infty$ and $p_0>3$, we obtain the desired bound.
\end{proof}

\begin{lem}\label{lem:recursive-weights}
Suppose that $\band_{T}\to0$ and $T_j\band_{T}\to\infty$. 
Then, for weights given by~(\ref{eq:recWeights}), Assumption~\ref{assump:weights} is satisfied with 
\begin{align}
  \label{eq:recursivedelta}
&\delta_{j,T}\sim (\band_{T}T_j)^{-1} \\
  \label{eq:recursiveV}
& V(i,v;i',v')=\pi\;2^{-i-i'},\quad i,i'\geq0,\;v\in\{0,\dots,2^{i}-1\},\;v'\in\{0,\dots,2^{i'}-1\}\;. 
\end{align}
\end{lem}
\begin{proof}
For convenience, we will omit the subscripts $_T$ and $_{j,T}$ in this proof when no ambiguity arises.
We set $u_j=[u\,T_j]$ in the following.
Using~(\ref{eq:rec_rho}), $\band T_j\to\infty$, $\band\to0$ and $u_j\sim u T_j$, we get that
\begin{equation} 
\label{eq:rec_rho_asymp}
\rho \sim (\band T_j) \;.
\end{equation}
Observing that $\delta=\gamma(u_j)=\rho^{-1}$, we get~(\ref{eq:recursivedelta}) and
Assumption~\ref{assump:weights}(\ref{item:delatAssump}) follows.

Let us now show that Assumption~\ref{assump:weights}(\ref{item:Vassump}) holds.
Using~(\ref{eq:ConvDisc}), we find that
\begin{align*}
\int_{-\pi}^\pi \Phi_{j,T}(\lambda;i,v)\overline{\Phi_{j,T}(\lambda;i',v')}\,\rmd\lambda
=&\frac{2\pi}{\rho_{j-i,T}\rho_{j-i',T}}
\exp\left(-\frac{u_{j-i}+v+1}{\band T_{j-i}}-\frac{u_{j-i'}+v'+1}{\band T_{j-i'}}\right)\\
&\times\,\sum_{l=0}^{N-1}\rme^{l\{2^i/(\band T_{j-i})+2^{i'}/(\band T_{j-i'})\}} \;,
\end{align*}
where $N=\{2^{-i}(u_{j-i}-v)\}\wedge\{2^{-i'}(u_{j-i'}-v')\}$.
Using~(\ref{eq:Tjasymp}),~(\ref{eq:rec_rho_asymp}), $\band T_j\to\infty$, $\band\to0$ and $u_j\sim u T_j$,
we obtain $\rho_{j-i,T}\sim 2^i(\band T_j)$, $(u_{j-i}+v+1)/(\band T_{j-i})\sim u/\band$,
$2^i/(\band T_{j-i})\sim 1/(\band T_j)$, $N2^i/(\band T_{j-i})\sim u/\band$
and similar result with $i',v'$ replacing $i,v$. Using these asymptotic equivalences and the previous display, we obtain  
\begin{equation} 
\label{eq:asympIntRec1}
\int_{-\pi}^\pi \Phi_{j,T}(\lambda;i,v)\overline{\Phi_{j,T}(\lambda;i',v')}\,\rmd\lambda
\sim\frac{2\pi}{2^{i+i'}(\band T_j)^2}\frac{A-o(1)}{2/(\band T_j)}\;,
\end{equation}
where 
$$
A=\exp\left(-\frac{u_{j-i}+v+1}{\band T_{j-i}}-\frac{u_{j-i'}+v'+1}{\band T_{j-i'}}
+N\left\{\frac{2^i}{\band T_{j-i}}+\frac{2^{i'}}{\band T_{j-i'}}\right\}\right)\;.
$$
Using~(\ref{eq:Tjasymp}), we have $N = uT_j + O(1)$ and $u_{j-i}+v+1=uT_j2^i+O(1)$. Thus
$N2^i-(u_{j-i}+v+1)=O(1)$ and the same holds with $i',v'$ replacing $i,v$. This implies that
$A=\exp\left(O\left((\band T_{j})^{-1}\right)\right)\to1$. This,~(\ref{eq:asympIntRec1}) and~(\ref{eq:recursivedelta}) yield
Assumption~\ref{assump:weights}(\ref{item:Vassump}) with $V(i,v;i',v')$ defined by~(\ref{eq:recursiveV}).

We finally show that Assumption~\ref{assump:weights}(\ref{item:neglicAssump}) holds.
By setting $N=2^{-i}(u_{j-i}+v)$ and $k=N-1-l$ in~(\ref{eq:defweightsFourier}), we obtain
$$
\left|\Phi(\lambda;i,v)\right|=\rho^{-1}\left|\sum_{k=0}^{N-1}\rme^{-k\{\rmi\lambda+2^i/(\band T_{j-i})\}}\right|
\leq\rho^{-1}\frac{1+\rme^{-N2^i/(\band T_{j-i})}}{\left|1-\rme^{-\rmi\lambda-2^i/(\band T_{j-i})}\right|}\;.
$$
Using that $N2^i/(\band T_{j-i})\sim\band^{-1}\to\infty$, $\delta^{-1/2}\rho^{-1}\to0$ and that, for any $\eta>0$, 
$|1-z|$ does not vanish on the compact set of complex numbers $z=r\rme^{\rmi\theta}$ such that $r\in[0,1]$ and 
$\eta\leq|\theta|\leq\pi$ and thus is lower bounded on this set, we obtain
Assumption~\ref{assump:weights}(\ref{item:neglicAssump}). 

Finally we show that Assumption~\ref{assump:weights}~(\ref{item:Locassump}) holds. By~(\ref{eq:Tjasymp}), we have, for any
$q\geq0$,
$$
\Gamma_q(u;j,T)=\rho^{-1}\sum_{k=0}^{u_j-1}\rme^{-(u_j-1-k)/(\band T_j)}|u_j-1-k|^q + O\left(\Gamma_0(u;j,T)\right)\;.
$$
Observe that $\Gamma_0(u;j,T)=1$. Setting $l=u_j-1-k$, and separating the above sum over $l\leq[q\band T_j]+1$ for which we
bound the exponential by 1 and
$l\geq[q\band T_j]+2$ so that $\rme^{-x/(\band T_j)}x^q$ is decreasing on $x\geq l-1$, we get
\begin{align*}
\sum_{k=0}^{u_j-1}\rme^{-(u_j-1-k)/(\band T_j)}|u_j-1-k|^q&\leq\sum_{l=0}^{[\band T_j]+1}l^q+\sum_{l\geq[\band
  T_j]+2}\rme^{-l/(\band T_j)}l^q\\
&\leq O\left((\band T_j)^{q+1}\right)+\int_{x\geq[q\band T_j]+1}\rme^{-x/(\band T_j)}x^q\rmd x\\
&=O\left((\band T_j)^{q+1}\right)\;.
\end{align*}
The last two displays,~(\ref{eq:rec_rho_asymp}) and~(\ref{eq:recursivedelta}) yield~(\ref{eq:NormalizedWeightsSums}), which
achieves the proof.
\end{proof}

\bibliographystyle{elsarticle-harv}
\bibliography{lrd}
\end{document}